\documentclass[12pt,a4paper]{article}
\pdfoutput=1
\usepackage[top=2cm, bottom=2cm, left=2cm, right=2cm]{geometry}
\usepackage{amsmath,amssymb,amsthm,mathtools}
\usepackage{natbib}
\usepackage{todonotes}
\usepackage{hyperref,url}
\usepackage{enumerate}
\usepackage{algorithm2e}
\hypersetup{
	colorlinks=true,
  linkcolor=red,          
  citecolor=blue,         
  filecolor=magenta,      
  urlcolor=cyan           
}
\counterwithout{equation}{section}

\newtheorem{theorem}{Theorem}[section]
\newtheorem{lemma}[theorem]{Lemma}
\newtheorem{proposition}[theorem]{Proposition}
\newtheorem{corollary}[theorem]{Corollary}

\theoremstyle{definition}

\newtheorem*{theorem*}{Theorem}
\newtheorem{definition}[theorem]{Definition}

\newtheorem{remark}[theorem]{Remark}

\newcommand{\pred}[1]{\boldsymbol{1}[#1]}

\newcommand{\set}[1]{\left\{ #1 \right\}}

\newcommand{\R}{\mathbb{R}}

\newcommand{\N}{\mathbb{N}}

\newcommand{\PP}{\mathbb{P}}
\newcommand{\pistar}{\pi_\star}

\newcommand{\Mirr}{\mathcal{M}_{\textnormal{irr}}}
\newcommand{\Mirrrev}{\mathcal{M}_{\textnormal{irr,rev}}}
\newcommand{\Mergrev}{\mathcal{M}_{\textnormal{erg,rev}}}

\newcommand{\dHel}{d_\textnormal{Hel}}
\newcommand{\dTV}{d_\textnormal{TV}}
\newcommand{\hitT}{\textnormal{HitT}}
\newcommand{\Cov}{\textnormal{Cov}}
\newcommand{\Distribution}{\textnormal{Distribution}}
\newcommand{\Distance}{\textnormal{Distance}}

\usepackage{xcolor}

\newcommand{\tblue}[1]{\textcolor{blue}{#1}}
\newcommand{\tred}[1]{\textcolor{red}{#1}}

\newcommand{\GW}[1]{\tred{\bf [GW: #1]}}
\newcommand{\SF}[1]{\tblue{\bf [SF: #1]}}

\DeclareMathOperator{\diag}{diag}

\DeclareMathOperator*{\argmin}{arg\,min}

\newcommand{\abs}[1]{\left| #1 \right|}

\newcommand{\nrm}[1]{\left\Vert #1 \right\Vert}

\newcommand{\PR}[2][]{\mathbb{P}_{#1}\left( #2 \right)}
\newcommand{\E}[2][]{\mathbb{E}_{#1}\left[ #2 \right]}

\newcommand{\eps}{\varepsilon}

\newcommand{\bigO}{\mathcal{O}}

\newcommand{\tmix}{t_{\mathsf{mix}}}
\newcommand{\trel}{t_{\mathsf{rel}}}

\usepackage{authblk}
\author[1]{Sela Fried}
\author[2]{Geoffrey Wolfer}
\affil[1]{Ben Gurion University of the Negev}
\affil[2]{JSPS International Research Fellow \protect\\  Department of Computer and Information Sciences \protect\\  Tokyo University of Agriculture and Technology }

\title{Identity Testing of Reversible Markov Chains}
\date{} 

\begin{document}

\maketitle

\begin{abstract}
We consider the problem of identity testing of Markov chain transition matrices based on a single trajectory of observations under the distance notion introduced by \citet{daskalakis2018testing} and further analyzed by \citet{pmlr-v99-cherapanamjeri19a}. Both works made the restrictive assumption that the Markov chains under consideration are symmetric. In this work we relax the symmetry assumption and show that it is possible to perform identity testing under the much weaker assumption of reversibility, provided that the stationary distributions of the reference and of the unknown Markov chains are close under a distance notion related to the separation distance. Additionally, we provide intuition on the distance notion of \citet{daskalakis2018testing} by showing how it behaves under several natural operations. In particular, we address some of their open questions.


\end{abstract}

\section{INTRODUCTION}

Efficiently distinguishing whether an unknown stochastic process is identical to a reference one or at least $\eps$-far from it under some notion of distance is a fundamental problem of the field of \emph{property testing}. Although the sample complexity in the iid case under the total variation distance is known to be of $\Theta(\sqrt{d}/\eps^2)$, where $d$ is the support size (see \citet{waggoner2015lp} for a summary), the Markovian case remains far from being settled.
In this setting one has the transition matrix of some reference Markov chain and needs to decide, with prescribed confidence, based on a single long trajectory sampled from an unknown Markov chain which started from an arbitrary state, whether the trajectory was sampled from the reference Markov chain or from a Markov chain that is at least $\eps$-far from it, with respect to some notion of distance.

Identity testing in the Markovian setting was applied by \citet{Daskalakis2017TestingFO} on the problem of testing whether an observed card shuffling is performed according to a certain riffle shuffle model. More recently, \citet{matsui2022analysis} used identity testing of Markov chains to analyse COVID-19 evolution.

\section{RELATED WORK AND MOTIVATION}

Prior work on the problem of identity testing of Markov chains within the property testing framework has so far focused on two main notions of distance: On one hand, \citet{daskalakis2018testing}, motivated by \citet{kazakos1978bhattacharyya}, considered a distance notion defined as the spectral radius of the entry-wise geometric mean of the transition matrices. Restricting their analysis to symmetric Markov chains, they obtained an upper bound of $\tilde\bigO(d / \eps + h )$ on the sample complexity where $d$ is the state space size, $\eps$ is the proximity parameter of the test, $h$ depends on the hitting time of the reference Markov chain  
and the tilde notation hides polylogarithmic factors in all quantities being used. They also proved a lower bound of $\Omega(d / \eps)$ and conjectured that this is the true sample complexity. Subsequently and still under the symmetry assumption,
\citet{pmlr-v99-cherapanamjeri19a} used sparsest-cut techniques 
to obtain an upper bound of $\tilde\bigO(d / \eps^4 )$ proving thereby that the sample complexity is independent of the hitting time. 

In parallel, 
\citet{pmlr-v108-wolfer20a} considered a distance  between Markov chains that relies on the infinity norm over stochastic matrices. They showed that, under this distance, the broader class of ergodic Markov chains may be addressed and proved an upper bound on the sample complexity of $\tilde{\Theta}((\sqrt{d}+\trel)/\pistar)$ where $\trel$ and $\pi_\star$ are, respectively, the relaxation time and the minimum stationary probability of the reference Markov chain.
Subsequently, \citet{chan2021learning} characterized the sample complexity of the problem in terms of what they refer to as the $k$-cover time which enabled them to generalize to all irreducible Markov chains. At the same time, \citet{fried2021alpha} too extended the results of \citet{pmlr-v108-wolfer20a} to irreducible Markov chains by moving to the $\alpha$-lazy versions of the Markov chains under consideration.



\section{MAIN CONTRIBUTION}
\label{section:main-contribution}

In this work we relax the symmetry assumption of \citet{daskalakis2018testing} and \citet{pmlr-v99-cherapanamjeri19a}, which poses a considerable limitation in terms of applicability. In the words of \citet[p. 3]{chan2021learning}, ``However, identity testing under this distance only works for symmetric Markov chains, which is a quite restricted sub-family of Markov chains [\dots]. Thus we do not study learning and testing problems under this distance."

Our main insight is that the methods of the aforementioned works are not limited to symmetric Markov chains and that a finer analysis 
allows to perform identity testing of Markov chains belonging to the much larger and natural class of reversible Markov chains. This generalization significantly enhances the applicability of the distance notion since the class of reversible Markov chains corresponds to random walks on networks (e.g. \citet[Section 9.1]{levin2017markov}).

A consequence of the symmetry assumption of \citet{daskalakis2018testing} and \citet{pmlr-v99-cherapanamjeri19a} is that the stationary distributions of the Markov chains under consideration are uniform. In particular, despite the uncertainty regarding the unknown Markov chain, we do know that its stationary distribution is equal to the stationary distribution of the reference Markov chain. In the generalization to reversible Markov chains this property is lost and that poses a challenge. Clearly, one possibility is to assume that the reference and the unknown Markov chains share the same (but now arbitrary) stationary distribution. 
However, this assumption can be difficult to verify empirically since stationary distributions could be arbitrarily close while being different. We show that with a more careful analysis and only at a constant price in terms of the sample complexity, it is possible to perform identity testing on reversible Markov chains whose stationary distributions are not too far apart from each other. 

Our main result is stated in the following theorem. It uses the distance notion between Markov chains of \citet{daskalakis2018testing} (cf. Definition \ref{definition:distance}) and of a distance notion between probability distributions that is closely related to the separation distance (cf. Definition \ref{def; sep}):  


\begin{theorem}
\label{theorem:upperbound-reversible}
Suppose we have the transition matrix $\bar{P}$ of an irreducible and reversible Markov chain with minimum stationary probability $\bar{\pi}_\star$ and let $\eps\in(0,1)$. There is a polynomial time algorithm which given access to a single trajectory of length $\tilde\bigO \left(1/\eps^4\bar{\pi}_\star\right)$ from an unknown irreducible and reversible Markov chain $P$ with a stationary distribution $\pi$ that satisfies $$\left|\left|\frac{\pi}{\bar{\pi}}-\mathbf{1}\right|\right|_{\infty}\leq\frac{\eps}{2},$$ distinguishes between $$P=\bar{P}\;\;\;\textnormal{ and }\;\;\; \Distance(P,\bar{P})\geq\eps,$$ with probability of success at least $3/5$.

\end{theorem}

\begin{remark}
We make the following observations:
\begin{enumerate}
    \item When the Markov chains are symmetric, $\bar{\pi}_\star=1/d$ and our upper bound of $\tilde\bigO \left(1/\eps^4\bar{\pi}_\star\right)$ takes the form of $\tilde\bigO \left(d/\eps^4\right)$, restoring the result of \citet[Theorem~10]{pmlr-v99-cherapanamjeri19a}. Thus, our generalization comes with no additional cost in terms of the sample complexity.
    \item Similar to the symmetric case, the upper bound of $\tilde\bigO \left(1/\eps^4\bar{\pi}_\star\right)$ involves only efficiently computable parameters of the  reference Markov chain and, in particular, does not depend on the unknown Markov chain.
    \item Having an upper bound on the mixing time of the unknown Markov chain significantly simplifies the problem. Indeed, suppose the mixing times of both the reference as well as the unknown Markov chain are upper bounded by $T$. By thinning the process and keeping pairs of observations $(X_t, X_{t+1})$ separated roughly by  $\tilde{\bigO}(T)$ steps, a trajectory of length $\tilde{\bigO}(T d / \eps)$ yields $\tilde{\bigO}(d / \eps)$ almost independent pairs. This suffices to perform identity testing on the edge measures with respect to the Hellinger distance which dominates the one in Definition~\ref{definition:distance}, allowing to solve the original problem by reduction. In particular, under this assumption, the requirement that the stationary distributions are close is redundant. Additionally, this shows that the dependency on $\pi_\star$ in our bound is not strictly necessary in this case. This work is concerned with the more challenging setting where the mixing time of the unknown Markov chain is unknown, precluding this approach.
    \item Both
    \citet{daskalakis2018testing} and \citet{ pmlr-v99-cherapanamjeri19a} make use of the hitting time to bound the amount of time needed to visit every state ``enough" times, which is, in the symmetric case, the same for all states. In contrast, in the reversible case, ``enough" depends on the weight of each state in the stationary distribution and it is not clear how the hitting time could be used in this case. We avoid the use of the hitting time by applying the concentration bounds of \citet{paulin2015concentration}.
    \item The algorithm of \citet{pmlr-v99-cherapanamjeri19a} has a polynomial time complexity and this is unchanged in our generalization.
\end{enumerate}

\end{remark}

\section{PRELIMINARIES}\label{sec; 1}
For $d\in\N$ we denote $[d]=\{1,2,\ldots,d\}$ and write $\Delta_d$ for the set of all probability distributions over $[d]$. 
Vectors will be written as row-vectors and for $x\in\R^d$ and $i\in[d]$ we write $x(i)$ for the $i$th entry of $x$. Similarly, if $P$ is a square matrix of size $d$ and $i,j\in[d]$, we write $P(i,j)$ for the entry at the $i$th row and the $j$th column of $P$. We refer to transition matrices of Markov chains simply as Markov chains.

\paragraph{Irreducible and reversible Markov chains.}

Denote by $\Mirr$ (resp. $\Mirrrev, \Mergrev$) the set of all irreducible (resp. irreducible and reversible, ergodic and reversible) Markov chains on the state space $[d]$. Let $P\in\Mirr,\mu\in\Delta_d, m\in\N$ and $i_1,\ldots,i_m\in[d]$. By $\left(X_t\right)_{t\in[m]}\sim(P,\mu)$ we mean  $$\PP((X_1,\ldots,X_m)=(i_1,\ldots,i_m))=\mu(i_1)\prod_{t=1}^{m-1}P(i_t,i_{t+1}).$$
For an entry-wise positive vector $\pi\in\R^d$ let $\langle\cdot,\cdot\rangle_\pi$ be the inner product on $\R^d$  given by $$\langle x,y\rangle_\pi =\sum_{i\in[d]}x(i)y(i)\pi(i),\;\;\forall x,y\in\R^d$$ (cf. \citep[p. 153]{levin2017markov}) and let $||\cdot||_{2,\pi}$ be the induced norm. We write $L_2(\pi) := \left(\R^d, \langle\cdot,\cdot\rangle_\pi\right)$ for the corresponding Hilbert space. For $P\in\Mirr$ with stationary probability $\pi$, the \emph{edge measure} $Q$ is defined by $Q=\diag(\pi)P$, where $\diag(\pi)$ is the diagonal matrix whose entries correspond to $\pi$ (cf. \citet[p. 88]{levin2017markov}). Define the \emph{time reversal of} $P$ by $P^* = \diag(\pi)^{-1}P^T\diag(\pi)$ (e.g. \citet[1.33]{levin2017markov}) and the \emph{multiplicative reversibilization of} $P$ by $P^\dagger = P^*P$ (e.g. \citet[2.2]{fill1991eigenvalue}). Finally, the \emph{spectral gap} $\gamma$ of $P\in\Mirrrev$ is defined by $\gamma=1-\lambda_2$ where $\lambda_2$ is the second largest eigenvalue of $P$ (cf. \citet[p. 154]{levin2017markov}).

\paragraph{Censored Markov chains.} 

Let $(X_t)_{t\in\N}$ be a Markov chain with transition matrix $P$ and let $\emptyset\neq S\subseteq [d]$. Consider the stochastic process $(X_{\tau_t})_{t\in\N}$ defined as follows: $\tau_1 = \inf\{i\in\N\;|\;X_i\in S\}$ and $\tau_{t+1} = \inf\{i\in\N \;|\; i > \tau_t, X_i \in S\}$ for every $t\in\N$. It is well known that $(X_{\tau_t})_{t\in\N}$ is a Markov chain. It is called the \emph{chain induced on $S$} \citep[Example~13.19]{levin2017markov} or the \emph{watched chain on $S$} \citep{levy1951,levy1952,levy1958} and we shall denote its transition matrix by $P_{\textnormal{cen}}(S)$. There is an explicit description of $P_{\textnormal{cen}}(S)$  in terms of certain submatrices of $P$ that uses the following notation: Let $R, T\subseteq [d]$. We write $P_{R,T}$ for the matrix obtained from $P$ by keeping only the rows and columns with indices in $R$ and $T$, respectively. If $R=T$ we write $P_R$ instead of $P_{R,R}$. With this notation it holds $$P_{\textnormal{cen}}(S) = P_S+ \sum_{t=1}^\infty P_{S,[d]\setminus S}P^t_{[d]\setminus  S}P_{[d]\setminus S,S}$$ (e.g. \citet[Lemma 6-6]{kemeny2012denumerable}). In addition, $P_{\textnormal{cen}}(S)$  is irreducible (resp. ergodic) if $P$ is irreducible (resp.  ergodic) and has the stationary distribution $\pi_S$ given by $$\pi_S(i)=\frac{\pi(i)}{\pi(S)},\;\;\forall i\in S$$ where $\pi(S)=\sum_{i\in S}\pi(i)$ (e.g. \citep[Lemma 2]{zhao1996censored}).
The following lemma generalizes to reversible Markov chains the statement in
 \citet[Lemma 14]{pmlr-v99-cherapanamjeri19a} according to which  $P_{\textnormal{cen}}(S)$ is symmetric if $P$ is symmetric. Its proof is given in Appendix~\ref{app; 1}.

\begin{lemma}\label{lem; 225}
Let $P\in\Mirrrev$ and let $S\subseteq [d]$. Then $P_{\textnormal{cen}}(S)\in\Mirrrev$.
\end{lemma}
\paragraph{Statistical distances.}

The Hellinger distance and the total variation distance between two distributions $p,q\in\Delta_d$ are  given by 
\begin{align}
\dHel^2(p,q) =&\frac{1}{2}\sum_{i\in[d]}\left(\sqrt{p(i)}-\sqrt{q(i)}\right)^2\nonumber \\ =& 1-\sum_{i\in[d]}\sqrt{p(i)q(i)}\nonumber\end{align}
and 

$$\dTV(p,q)=\frac{1}{2}\sum_{i\in[d]}|p(i)-q(i)|,$$ respectively.

It is well known (e.g. \citet[p. 13]{gibbs2002choosing}) that
\begin{equation}\label{eq; 46}
\dHel^2(p,q)\leq\dTV(p,q)\leq\sqrt{2}\dHel(p,q).
\end{equation}

For $P,\bar{P}\in\Mirr$ we denote by $P\circ\bar{P}$ the \emph{Hadamard product of $P$ and $\bar{P}$} (e.g. \citet[Definition 7.5.1]{horn2012matrix}) and by $\sqrt{P}$ the square matrix of size $d$ such that $$\sqrt{P}(i,j)=\sqrt{P(i,j)},\;\;\forall i,j\in[d].$$ 

For $p,q\in\Delta_d$, we define $p\circ q$ and $\sqrt{p}$ analogously.

The following distance between two Markov chains was proposed by \citet[p. 4]{daskalakis2018testing} who motivated it by noticing that (a) it vanishes if and only if the Markov chains share an identical essential communicating class and (b) it captures the ability to distinguish between the Markov chains based on a single long trajectory \citep[Claims 1 and 2]{daskalakis2018testing}. The distance relies on the \emph{spectral radius} $\rho(P)$ of a square matrix $P$ defined by $$\rho(P)=\max\{|\lambda|\;|\;\lambda \text{ is an eigenvalue of } P\}$$ (e.g. \citet{horn2012matrix}).

\begin{definition}
\label{definition:distance}
Let $P$ and $\bar{P}$ be two Markov chains. Define $$\Distance(P,\bar{P})=1-\rho\left(\sqrt{P\circ \bar{P}}\right).$$
\end{definition}

We will perform identity testing on reversible Markov chains whose stationary distributions are not too far from each other under the following distance notion, which is closely related to the \emph{separation distance} (e.g. \citet[6.7]{levin2017markov} or \citet[p. 72]{aldous1987strong}):

\begin{definition}\label{def; sep}
Let $\mu,\nu\in\Delta_d$ such that $\nu$ is entry-wise positive. Denote by $\mathbf{1}$ the vector in $\mathbb{R}^d$ that has all entries equal to $1$. Define
$$\left|\left|\frac{\mu}{\nu}-\mathbf{1}\right|\right|_{\infty} := \max_{i\in[d]}\left|\frac{\mu(i)}{\nu(i)}-1\right|.$$
\end{definition}


\section{STATE SPACE PARTITIONING AND COMPONENT ANALYSIS}
\label{section:space-partition}

The main insight of \citet{{pmlr-v99-cherapanamjeri19a}} that enables them to improve upon \citet{daskalakis2018testing} and discard the dependency on the hitting time is that in order to distinguish between two different Markov chains it is sufficient to analyse trajectories in subsets of states that are internally well connected (see Corollary \ref{cor; 66} and the paragraph preceding it). To achieve this they devise a new algorithm for partitioning of the state space $[d]$. This algorithm, upon receiving a reference Markov chain and a tolerance parameter, returns a tuple $(\mathcal{S}, T)$ where $\mathcal{S}$ is a set of well connected subsets of $[d]$ (components) and $T$ is a subset of $[d]$ in which the Markov chain does not spend too much time. The subsets in $\mathcal{S}$, together with $T$, form a partition of $[d]$. The Markov chain identity testing is then reduced to iid identity testing of distributions induced by these components. In this section we introduce the machinery and prove its properties. 



\subsection{Spectral and probabilistic properties of the components}
\label{section:spectral-properties-components}



Recall that for $P\in \Mirrrev$ and $\nu\in\Delta_d$ the matrix $\diag(\nu)P$ encodes a distribution over $[d]^2$. If $\nu = \pi$, where $\pi$ is the stationary distribution of $P$, the joint distribution $\diag(\nu)P$ corresponds to the edge measure $Q$. To some component $S\subseteq [d]$ we wish to similarly associate a probability distribution over which we will later apply iid identity testing. To this end, we first consider $\diag(\nu_S)P_S$. Second, we merge all the outgoing transitions from $S$, i.e., all $(i,j)\in S\times [d]$ such that $j\notin S$, into a single symbol denoted by $\infty$. This is the idea behind the following

\begin{definition}
Let $P$ be a Markov chain and let $\nu\in\Delta_d$. For $S\subseteq [d]$ such that $\nu(S)>0$ we denote by $\Distribution(S, P, \nu)$ the probability distribution on the set $S^2\cup\{\infty\}$ defined as follows:
\begin{align}\Distribution(S, P,\nu)(i,j)=&\frac{\nu(i)P(i,j)}{\nu(S)},\;\;\forall i,j\in S \textnormal{ and}\nonumber\\\Distribution(S, P, \nu)(\infty) =& \nonumber\\ 1- \sum_{i,j\in S}& \Distribution(S, P, \nu)(i,j).\nonumber\end{align}
\end{definition}

Let $P,\bar{P}\in\Mirrrev$ with stationary distributions $\pi,\bar{\pi}$, respectively, where we assume that $\bar{P}$ is given and that $P$ is unknown. A key property is that a positive distance between $P$ and $\bar{P}$ results in a positive Hellinger distance between the probability distributions induced by the corresponding edge measures over a component $S$, provided that in $\bar{P}$ there is enough weight on the transitions between states in $S$ (cf. \citet[Lemma 13]{pmlr-v99-cherapanamjeri19a}). But, since in the reversible case the stationary distributions are, in general, not uniform, it is not clear how to sample from the unknown edge measure. We solve this problem by sampling from the joint distribution $R=\diag(\bar{\pi})P$. This approach is guaranteed to succeed if $\pi$ is not too far from $\bar{\pi}$. This is the content of the following lemma. A sketch of its proof is given here while the full proof is given in Appendix \ref{section:proof-min-max}. 

\begin{lemma}\label{lem; 537}
Let $P,\bar{P}\in\mathcal{M}_{\textnormal{irr,rev}}$ with stationary distributions $\pi,\bar{\pi}$, respectively. Let $\eps\in(0,1)$ be such that $\textnormal{Distance}(P, \bar{P})\geq\eps$ and $$\left|\left|\frac{\pi}{\bar{\pi}}-\mathbf{1}\right|\right|_{\infty}\leq\frac{\eps}{2}.$$ Let $S\subseteq [d]$ such that $$\sum_{i,j\in S} \textnormal{Distribution}(S, \bar{P},\bar{\pi})(i,j)\geq 1-\frac{\eps}{16}.$$ Then \begin{align}d^2_\textnormal{Hel}(\textnormal{Distribution}(S, P,\bar{\pi}), \textnormal{Distribution}(S, \bar{P}, \bar{\pi}))&\geq\nonumber \\ &\frac{\eps^2}{128}.\nonumber\end{align}
\end{lemma}

\begin{proof}[Proof sketch]
Let $\bar{Q}=\diag({\bar{\pi}})\bar{P}$ and $R=\diag({\bar{\pi}})P$. From the reversibility of $P,\bar{P}$ it follows that $\sqrt{P \circ \bar{P}}$ is self-adjoint in $L_2\left(\sqrt{\pi \circ \bar{\pi}}\right)$. Applying the Courant–Fischer principle (e.g. \citet[p. 219]{helmberg2008introduction}), taking $u$ to be the characteristic function of $S$ and using the AM-GM inequality we obtain

\begin{align}
\rho&\left(\sqrt{P\circ \bar{P}}\right) = \max_{u\neq 0}\frac{\langle \sqrt{P \circ \bar{P}}u,u\rangle_{\sqrt{\pi\circ \bar{\pi}}}}{||u||_{2,\sqrt{\pi \circ \bar{\pi}}}^2} \nonumber\\ &\geq  \sum_{i,j\in S}\sqrt{\frac{R(i,j)}{\bar{\pi}(S)}} \sqrt{\frac{\bar{Q}(i,j)}{\bar{\pi}(S)}}\sqrt{\frac{\pi(i)}{\bar{\pi}(i)}}\sqrt{\frac{\bar{\pi}(S)}{\pi(S)}}.\nonumber
\end{align}
The assumption on $\frac{\pi}{\bar{\pi}}$ guarantees that $\sqrt{\frac{\pi(i)}{\bar{\pi}(i)}}\sqrt{\frac{\bar{\pi}(S)}{\pi(S)}}$ is not too small. The claim follows by distinguishing between two cases.
\end{proof}


Connectivity (or conductance) measures how a Markov chain navigates among its states and is controlled by the Cheeger constant (cf. \citet[p. 88]{levin2017markov}):

\begin{definition}\label{def; 13}
Let $P\in\Mirrrev$ with stationary distribution $\pi$ and edge measure $Q$. Let $I\subseteq [d]$ and $S\subsetneqq I$. The \emph{bottleneck ratio of $S$ in $I$} is defined by $$\Phi(P, S, I)=\frac{\sum_{i\in S, j\in I\setminus S}Q(i,j)}{\min\{\pi(S), \pi(I\setminus S)\}}$$ and the \emph{Cheeger constant of $P$} is defined by $$\Phi_\star(P)=\min_{S\subsetneqq [d]}\Phi(P, S, [d]).$$
\end{definition}

We shall refer to a component $S\subseteq [d]$ for which the condition in the following corollary is satisfied with some $\alpha>0$ as \emph{well connected}. A well connected component allows to control the spectral gap of the induced censored Markov chain: 

\begin{corollary}\label{cor; 66}
Let $P\in\Mirrrev$ with stationary distribution $\pi$. Let $S\subseteq [d]$ and $\alpha\geq 0$ be such that 
$$\Phi(P, R, S)\geq\alpha,\;\;\forall R\subsetneqq S.$$ 
Then $\gamma\geq\frac{\alpha^2}{2}$ where  $\gamma$ is the spectral gap of $P_{cen}(S)$.
\end{corollary}

\begin{proof}
Let $Q$ be the edge measure of $P$ and suppose $R\subsetneqq S$ is such that $\pi(R)\leq\frac{\pi(S)}{2}$. Since $P_{\textnormal{cen}}(S)(i,j)\geq P(i,j)$ for every $i,j\in S$, we obtain 
\begin{equation*}
    \begin{split}
        \Phi(P_{\textnormal{cen}}(S), R, S) &=\frac{\sum_{i\in R, j\in S\setminus R} \frac{\pi(i)}{\pi(S)}P_{\textnormal{cen}}(S)(i,j)}{\pi(R)/\pi(S)}\\ &\geq \frac{\sum_{i\in R,j\in S\setminus R} Q(i,j)}{\pi(R)} \\ & =\Phi(P, R, S) \geq\alpha.        
    \end{split}
\end{equation*}
Thus, $\Phi_\star(P_{\textnormal{cen}}(S))\geq\alpha$. By Cheeger's inequality
\citep[Lemma 3.3]{sinclair1989approximate}, the spectral gap $\gamma$ of $P_{\textnormal{cen}}(S)$ satisfies
$\gamma \geq \Phi_\star(P_{\textnormal{cen}}(S))^2/2$.

\end{proof}



Recall that the state space partitioning algorithm, in addition to the well connected components, returns a subset $T\subseteq [d]$ comprising of states that belong to no component. The following lemma bounds the largest eigenvalue of the submatrix $P_T$. Its proof is a variation of the proof of \citep[Lemma 3.3]{sinclair1989approximate} and is given in Appendix \ref{app; 2}. The bound is then used in Lemma \ref{lem; 890} to upper bound the time the Markov chain spends in $T$.
The proof of Lemma \ref{lem; 890} is given in Appendix \ref{section:proof-lemma-890}.

\begin{lemma}
\label{lem; 32}
Let $P\in\Mirrrev$ with stationary distribution $\pi$ and edge measure $Q$. Let $T\subsetneqq [d]$ and $\alpha\geq 0$ be such that $$\frac{\sum_{i\in R,j\in[d]\setminus R}Q(i,j)}{\pi(R)}\geq\alpha, \;\;\forall R\subseteq T.$$ Let $\lambda$ denote the largest eigenvalue of $P_T$. Then $$\lambda\leq 1-\frac{\alpha^2}{2}.$$ 
\end{lemma}


\begin{lemma}\label{lem; 890}
Let $P\in\Mirrrev$ with stationary distribution $\pi$ and edge measure $Q$. Let $T\subsetneqq [d]$ and $\alpha> 0$ be such that $$\frac{\sum_{i\in R,j\in [d]\setminus R} Q(i,j)}{\pi(R)}\geq\alpha, \;\;\forall R\subseteq T.$$ Denote $(\pi_T)_\star = \min_{i\in T}\{\pi(i)\}$. Let $\delta\in(0,1)$ and suppose $m\geq\Omega\left( \frac{\log\frac{1}{(\pi_T)_\star}\log\frac{1}{\delta}}{\alpha^2}\right)$. Then $$\PP\left(\sum_{t=1}^m 1\{X_i\notin T\}\geq \Omega\left(\frac{m\alpha^2}{\log \frac{1}{(\pi_T)_\star}}\right)\right)\geq 1-\delta.$$
\end{lemma}

\subsection{Markov chain partitioning algorithm}\label{sec; alg}

In \citet{pmlr-v99-cherapanamjeri19a} an algorithm for partitioning the state space $[d]$ is devised which, upon receiving a reference Markov chain and a tolerance parameter, returns two objects: A set $\mathcal{S}$ of subsets of $[d]$ in which the iid identity tester will be applied and $T\subseteq [d]$ in which the Markov chain does not spend too much time. The properties of these two objects are given in the following theorem. The algorithm (and its analysis) needs to be modified in order to be applicable beyond the symmetric case. In the rest of the section we give only the details whose modification was more involved. The reader is referred to \citet[Appendix A]{pmlr-v99-cherapanamjeri19a} for the complete description and analysis of the algorithm.


\begin{theorem}\label{thm; 99}
Let $P\in\Mirrrev$ with stationary distribution $\pi$ and edge measure $Q$. Let $\beta\in (0,1)$. There exists an algorithm that returns a tuple $(\mathcal{S}, T)$ such that $S\subseteq [d],\forall S\in\mathcal{S}, T\subseteq [d]$ and $\mathcal{S}\cup\{T\}$ is a partition of $[d]$. Furthermore, for every  $S\in\mathcal{S}$ it holds:
\begin{enumerate}
    \item [(1)] $\frac{\sum_{i,j\in S}Q(i,j)}{\pi(S)}\geq 1-\beta$.
    \item [(2)] 
    $\Phi(P,R,S)\geq\Omega\left(\frac{\beta}{\log^2 d}\right),\;\;\forall R\subsetneqq S$.
    \item [(3)] $\frac{\sum_{i\in R,j\in [d]\setminus R}Q(i,j)}{\pi(R)}\geq\Omega\left(\frac{\beta}{\log d}\right),\;\;\forall R\subseteq T$.
\end{enumerate}
\end{theorem}

\begin{definition}
Let $S\subsetneqq[d]$. The \emph{cut metric associated with $S$} is defined by $$\delta_{S}(i,j)=
\begin{cases} 
0 & \text{if } i,j\in S \text{ or } i,j\in [d]\setminus S\\
1 &\text{otherwise}.
\end{cases}$$
\end{definition}


\begin{definition}
Let $P\in\Mirrrev$ with stationary distribution $\pi$ and edge measure $Q$. Let $I\subseteq [d]$ and $T\subsetneqq I$. The \emph{sparsest cut with component constraints (SPCCC)} is defined as $$S^* = \argmin_{T\subseteq S\subsetneqq I}\frac{\sum_{i,j\in I}Q(i,j)\delta_S(i,j)}{\sum_{i,j\in I}\pi(i)\pi(j)\delta_S(i,j)}.$$

The corresponding linear programming relaxation is given by
$$\min\sum_{i,j\in I}Q(i,j)\delta_{ij}\textnormal{ such that}$$
$$\delta_{ii}=0,\;\;\forall i\in I$$ $$\delta_{ij}\leq\delta_{ik}+\delta_{kj}, \;\;\forall i,j,k\in I$$ $$\sum_{i,j\in I}\pi(i)\pi(j)\delta_{ij}=1$$ $$\delta_{ij}\geq 0$$ \begin{equation}\label{eq; 00}\delta_{ij}=0,\;\;\forall i,j\in T\end{equation} $$\delta_{ik}=\delta_{jk}, \;\;\forall i,j\in T, k\in I.$$
\end{definition}

\begin{theorem}\label{thm; 11}
Given an instance of the $\textnormal{SPCCC}$ problem there exists a polynomial time algorithm $\textnormal{FindComp}$ that returns $S'\subsetneqq I$ such that $S'\cap T=\emptyset$ and \begin{equation}\label{eq; 111}
\begin{split}
&\frac{\sum_{i,j\in I}Q(i,j)\delta_{S'}(i,j)}{\sum_{i,j\in I}\pi(i)\pi(j)\delta_{S'}(i,j)} \\& \leq O(\log d) \min_{T\subseteq S\subsetneqq I}\frac{\sum_{i,j\in I}Q(i,j)\delta_S(i,j)}{\sum_{i,j\in I}\pi(i)\pi(j)\delta_S(i,j)}.
\end{split}
\end{equation}
\end{theorem}

\begin{proof}
The proof of \citet[Theorem 4.1]{linial1995geometry} literally extends to our case (the symmetry of $Q(i,j)$ due to reversibility and the symmetry of $\pi(i)\pi(j)$ are crucial) and establishes (\ref{eq; 111}) for some $S'\subsetneqq I$. We shall provide only the details necessary to prove the claim regarding the intersection: First, notice that it suffices to prove that either $S'\cap T=\emptyset$ or $\left(I\setminus S'\right)\cap T=\emptyset$ since the left hand-side of (\ref{eq; 111}) is the same for both $S'$ and $I\setminus S'$ and we let the algorithm return the set satisfying the intersection condition. Now, denote by $\delta$ the minimizing metric that is found by the linear program above. By Bourgain’s metric-embedding theorem (e.g. \citet[Corollary 3.4]{linial1995geometry}), there are $x_1,\ldots,x_d\in\R^m$ where $m=O(\log^2 d)$ such that for every $i,j\in I$ it holds $$\Omega\left(\frac{1}{\log d}\right)\delta_{ij}\leq ||x_i - x_j||_1\leq \delta_{ij}.$$ It follows from this, together with the constraint (\ref{eq; 00}), that $x_i=x_j$ for every $i,j\in T$. Recall that by the proof of \citet[Theorem 4.1]{linial1995geometry}, $S'=\{i\in I\;| \;x_i(r)=1\}$ for some $r\in[m]$. Conclude that either $T\subseteq S'$ or $T\subseteq I\setminus S'$.
\end{proof}

\begin{corollary}
In the setting of Theorem \ref{thm; 11}, for $S'$ returned by the $\textnormal{FindComp}$ algorithm it holds $$\Phi(P, S', I)\leq O(\log d)\min_{T\subseteq S\subsetneqq I}\Phi(P, S, I).$$
\end{corollary}

\begin{proof}
Let $S\subsetneqq I$. It is easy to see that \begin{equation}\label{eq; 344}\frac{\sum_{i,j\in I}Q(i,j)\delta_{S}(i,j)}{\sum_{i,j\in I}\pi(i)\pi(j)\delta_{S}(i,j)}=\frac{\sum_{i\in S,j\in I\setminus S}Q(i,j)}{\pi(S)\pi(I\setminus S)}.\end{equation}
Now, assume that $\frac{\pi(I)}{2}\leq\pi(S)\leq\pi(I)$. Then 
\begin{equation*}
    \begin{split}
      1 \leq \frac{\pi(I) \min\{\pi(S),\pi(I\setminus S)\}}{\pi(S)\pi(I\setminus S)}\leq 2.
    \end{split}
\end{equation*}
By symmetry, this holds also if $\frac{\pi(I)}{2}\leq\pi(I\setminus S)\leq\pi(I)$. Conclude
that \begin{equation}\label{eq; 23}
\begin{split}
\frac{\pi(I)\sum_{i\in S,j\in I\setminus S}Q(i,j)}{2\pi(S)\pi(I\setminus S)}&\leq\Phi(P, S, I)\\&\leq \frac{\pi(I)\sum_{i\in S,j\in I\setminus S}Q(i,j)}{\pi(S)\pi(I\setminus S)}.
\end{split}\end{equation}
The claim follows by combining (\ref{eq; 111}), (\ref{eq; 344}) and (\ref{eq; 23}) together.
\end{proof}

\subsection{Sample generation}\label{sec; 788}

Both \citet{daskalakis2018testing} and \citet{pmlr-v99-cherapanamjeri19a} reduce the Markov chain identity testing to identity testing of  distributions in the iid case and use a black box tester which we will refer to as $\textnormal{iidTester}$ (cf. \citep[Algorithm 1]{daskalakis2018distribution}). The algorithm  accepts the arguments $(Y_t)_{t\in[m]}, \bar{p}, \eps,\delta$ and returns either $0$ or $1$. The guarantees on the tester's performance as well as the meaning of each of its arguments is given in the following lemma, which is used in Lemma \ref{lem; 356}.

\begin{lemma}\label{lem; 54}
Suppose we have the description of a probability distribution $\bar{p}\in\Delta_d$ and let $\eps,\delta\in(0,1)$. Then there is a tester called $\textnormal{iidTester}$ such that if it is given $m\geq \Omega\left(\frac{\sqrt{d}\log\frac{1}{\delta}}{\eps^2}\right)$ iid samples $(Y_t)_{t\in[m]}$ from an unknown probability distribution  $p\in\Delta_d$ the following holds with probability at least $1-\delta$: 
\begin{align}
\textnormal{iidTester}\left((Y_t)_{t\in[m]},\bar{p}, \eps, \delta\right) =
\begin{cases}
0 & \textnormal{if } p=\bar{p} \\
1 & \textnormal{if } \dHel(p,\bar{p})\geq\eps.
\end{cases}\nonumber
\end{align}
\end{lemma}


The iid sampling process from the (Markovian) observed trajectory comprises of two stages: First, sample from $[d]$ according to the stationary distribution $\bar{\pi}$ of the reference Markov chain. Second, record the transitions from these states in the trajectory (cf. \citet[pp. 12-13]{daskalakis2018testing}). This process is \citet[Algorithm 1]{pmlr-v99-cherapanamjeri19a} and will be referred to as $\textnormal{iidGenerator}$. Notice that when it is applied on $S\subseteq [d]$, instead of sampling uniformly and iid from $S$, which corresponds to the symmetric case, it needs to sample from $S$ according to the distribution $\bar{\pi}_S$. We elaborate on the iidGenerator  algorithm in Appendix \ref{section:q-sampler}.

Following \citet{daskalakis2018testing},  \citet[Lemma 19]{pmlr-v99-cherapanamjeri19a} use a bound involving the hitting time to guarantee enough visits of each state in a component. 
It is not immediately clear how this approach generalizes to the reversible case. We, instead, exploit the concentration bounds of \citet{paulin2015concentration} that rely on the spectral gap. This comes at no additional cost and arguably simplifies the analysis.

The following two lemmas give guarantees for the success of  this iid sampling process. Their proofs are given in Appendix~\ref{section:proof-paulin} and \ref{522}.

\begin{lemma}\label{lem; 1}
Let $P\in\Mergrev$ with stationary distribution $\pi$ and spectral gap $\gamma$. Let $\mu\in\Delta_d$ and  $\delta\in(0,1)$. Assume $(X_t)_{t\in\N}\sim(P,\mu)$. For $m\in\N$ and $i\in[d]$ denote by $N_m(i) = \sum_{t\in[m]}1\{X_t=i\}$ the number of occurrences of state $i$ in $(X_t)_{t\in[m]}$. Suppose   $m\geq\Omega\left(\frac{\log \frac{1}{\delta\pistar}}{\pistar\gamma}\right)$. Then 
$$\PP\left(N_m(i)\geq \frac{\pi(i)}{2}m, \;\;\forall i\in[d]\right)\geq 1-\delta.$$
\end{lemma}

\begin{lemma}\label{lem; 522}
Let $m\in\N$ and suppose $\left(Y_t\right)_{t\in[m]}$ are iid sampled according to a distribution $p\in\Delta_d$. Let $v\in\{0,1,\ldots,m\}^d$ be the histogram of $\left(Y_t\right)_{t\in[m]}$ and let $\delta\in(0,1)$. Assume $m\geq \Omega\left(\frac{\log\frac{d}{\delta}}{p_{\star}}\right)$. Then $$\PP\left(v(i)\leq 2mp(i),\;\;\forall i\in[d]\right)\geq 1-\delta.$$
\end{lemma}

In Lemma \ref{lem; 356} it is assumed that there is access to enough iid samples from $\Distribution(S, P, \bar{\pi})$. The following lemma upper bounds the number of visits to $S$ that guarantees this. Due to the way the identity tester of the Markov chains is defined, such a guarantee is only needed in the case that $P=\bar{P}$. 

\begin{lemma}\label{lem; 222}
Let $P\in\Mergrev$ with stationary distribution $\pi$. Let $\eps\in(0,1)$ and let $(\mathcal{S}, T)$ be the tuple returned by the state space partitioning algorithm (Theorem \ref{thm; 99}) applied on $P$ and $\beta = \eps$. Let $\mu\in\Delta_d$ and suppose $(X_t)_{t\in\N}\sim (P,\mu)$. Then it is possible, with probability at least $\frac{9}{10}$, to generate $\frac{m}{4}$ iid samples from $\Distribution(S, P, \pi)$, for every $S\in\mathcal{S}$ that is visited at least  $m\geq\Omega\left(\frac{\log^4 d\log \frac{d}{(\pi_S)_\star}}{\eps^2(\pi_S)_\star}\right)$ times.
\end{lemma}

\begin{proof}
Let $S\in\mathcal{S}$. By Theorem \ref{thm; 99} (2), $$\Phi(P,R,S)\geq\Omega\left(\frac{\eps}{\log^2 d}\right),\;\;\forall R\subsetneqq S.$$ By Corollary \ref{cor; 66}, $ \gamma\geq \Omega\left(\frac{\eps^2}{\log^4 d}\right)$ where $\gamma$ is the spectral gap of $P_{cen}(S)$. Thus, by Lemma \ref{lem; 1} applied on $P_{cen}(S)$ and $\delta=\frac{1}{20d}$, in a trajectory of length $m\geq\Omega\left(\frac{\log^4 d\log\frac{20d^2}{(\pi_S)_\star}}{\eps^2(\pi_S)_\star}\right)$ each state $i\in S$ is visited at least $\frac{\pi(i)}{2\pi(S)}m$ times, with probability at least $1-\frac{1}{20d}$. By Lemma \ref{lem; 522}, applied on $p=\pi_S$ and $\delta=\frac{1}{20d}$, if $m$ satisfies the above inequality, then $$\PP\left(v(i)\leq 2m\frac{\pi(i)}{\pi(S)},\;\;\forall i\in S\right)\geq 1-\frac{1}{20d}.$$ Intersecting the above two events, we conclude that, with probability at least $1-\frac{1}{10d}$, we can to generate $\frac{m}{4}$ iid samples from $\Distribution(S, P, \pi)$. By the union bound, with probability at least $\frac{9}{10}$, this is true for every $S\in\mathcal{S}$ that is visited at least $m$ times. \end{proof}

\subsection{Reduction to the iid case}

The following lemma shows that in order to perform identity testing of Markov chains it suffices to perform identity testing in the iid case over any component in $\mathcal{S}$ that is visited enough times:

\begin{lemma}\label{lem; 356}
Let $P, \bar{P}\in\Mirrrev$ with stationary distributions $\pi,\bar{\pi}$, respectively, and let $\eps\in(0,1)$. Assume that $$\left|\left|\frac{\pi}{\bar{\pi}}-\mathbf{1}\right|\right|_{\infty}\leq\frac{\eps}{2}.$$ Let $(\mathcal{S}, T)$ be the tuple returned by the state space partitioning algorithm (Theorem \ref{thm; 99}) applied on $\bar{P}$ and $\beta = \frac{\eps}{16}$. Let $\mu\in\Delta_d$ and suppose $(X_t)_{t\in\N}\sim (P, \mu)$. Let $S\in\mathcal{S}$ and assume that we have $m\geq \Omega\left(\frac{|S|\log d}{\eps^2}\right)$ iid samples $(Y_t)_{t\in[m]}$ from   $\Distribution(S, P, \bar{\pi})$. Then, with probability at least $\frac{9}{10}$,
\begin{equation*}
    \begin{split}
     \textnormal{iidTester}\Bigg((Y_t)_{t\in[m]},\;&  \Distribution(S, \bar{P}, \bar{\pi}), \frac{\eps^2}{128}, \frac{1}{10d}\Bigg) \\&=
\begin{cases}
0 & \textnormal{if } P=\bar{P} \\
1 & \textnormal{if } \Distance(P,\bar{P})\geq\eps.
\end{cases}
    \end{split}
\end{equation*}
\end{lemma}

\begin{proof}
Let $\bar{Q}$ be the edge measure of $\bar{P}$. By Theorem \ref{thm; 99} (1), \begin{align} \sum_{i,j\in S}\Distribution(S, \bar{P},\bar{\pi})(i,j)=&\frac{\sum_{i,j\in S}\bar{Q}(i,j)}{\pi(S)}\nonumber \\\geq& 1-\frac{\eps}{16}.\nonumber\end{align}
By Lemma \ref{lem; 537}, \begin{align}d^2_\textnormal{Hel}(\Distribution(S, P,\bar{\pi}), \Distribution(S, \bar{P},\bar{\pi}))&\geq \nonumber\\&\frac{\eps^2}{128}\nonumber.\end{align} Since $|\mathcal{S}|\leq d$, the claim follows from a union bound and the guarantees of Lemma \ref{lem; 54}.
\end{proof}

\section{PROOF OF THEOREM \ref{theorem:upperbound-reversible}}
\label{section:proof-main-result}

To avoid periodicity issues, we exploit the fact (Proposition \ref{prop, 112} (\ref{112A})) that the distance between two Markov chains does not change much if the Markov chains are replaced by their respective $\alpha$-lazy versions, as long as $\alpha$ is not too big. Thus, we may assume that $P,\bar{P}$ are ergodic, at the cost of replacing $\eps$ with $\frac{\eps}{2}$. 
Let $(\mathcal{S}, T)$ be the tuple returned by the state space partitioning algorithm (Theorem \ref{thm; 99}) applied on $P$ and $\beta = \eps$.
First assume that $P=\bar{P}$. Let $Q$ be the edge measure of $P$. By Theorem \ref{thm; 99} (3), $$\frac{\sum_{i\in R, j\in  [d]\setminus R}Q(i,j)}{\pi(R)}\geq\Omega\left(\frac{\eps}{\log d}\right),\;\;\forall R\subseteq T.$$ Applying Lemma \ref{lem; 890} with $\alpha=\frac{\eps}{\log d}$ and $\delta=\frac{1}{10}$ we have that if $m\geq\Omega\left(\frac{\log^6d\log\frac{1}{\pistar}\log \frac{d}{\pi_\star}}{\eps^4\pistar}\right)$ then $$\PP\left(\sum_{t=1}^m 1\{X_t\in \bigcup_{S\in\mathcal{S}}S\}\geq\Omega\left( \frac{\log^4d\log \frac{d}{\pi_\star}}{\eps^2\pistar}\right)\right)\geq\frac{9}{10}.$$ It holds $$\sum_{S\in\mathcal{S}}\frac{1}{(\pi_S)_\star} = \sum_{S\in\mathcal{S}}\frac{\pi(S)}{\min_{i\in S}\{\pi(i)\}}\leq\sum_{S\in\mathcal{S}}\frac{\pi(S)}{\pistar}= \frac{1}{\pistar}.$$ Thus, by the strong form of the pigeonhole principle (e.g. \citet[Theorem 2.2.1]{brualdi1977introductory}), at least one $S\in\mathcal{S}$ is visited more than
$\Omega\left(\frac{\log^4d\log \frac{d}{\pi_\star}}{\eps^2(\pi_S)_\star}\right)$ times. By Lemma \ref{lem; 222}, with probability at least $\frac{9}{10}$, we may sample $$\Omega\left(\frac{\log^4d\log \frac{d}{\pi_\star}}{\eps^2(\pi_S)_\star}\right)\geq
\Omega\left(\frac{|S|\log d}{\eps^2}\right)$$ iid samples from $$\Distribution(S, P, \bar{\pi})=\Distribution(S, \bar{P}, \bar{\pi}).$$ 
Applying Lemma \ref{lem; 356} we identify that $P=\bar{P}$ with probability of success at least $\frac{9}{10}$. We conclude that, in this case, the tester succeeds with probability at least $\frac{7}{10}$.

Assume now that $\Distance(P, \bar{P})\geq\eps$. In this case the algorithm fails only if the iid sampling succeeds in at least one component $S\in\mathcal{S}$ but the iid tester fails on $S$. The probability that this happens is upper bounded by $\frac{1}{10}$ (Lemma \ref{lem; 54}).
\qed

\section{SOME PROPERTIES OF THE DISTANCE}

The purpose of this section is to develop intuition for the distance notion considered in this work (Definition~\ref{definition:distance}). This will be done by establishing its behaviour under several natural operations, two of which were inspired by \citet[Open Questions 2 and 3]{daskalakis2018testing}. The proof of the following proposition is given in Appendix \ref{sss} .


\begin{proposition}\label{prop, 112}
Let $P,\bar{P}\in\Mirr$ with stationary distributions $\pi,\bar{\pi}$, time reversals $P^*, \bar{P}^*$ and multiplicative reversibilizations $P^\dagger,\bar{P}^\dagger$, respectively. Let $\eps\in(0,1)$.
\begin{enumerate}[$(i)$]
    \item\label{112A} Assume that $\Distance(P, \bar{P})\geq \eps$. Let $\alpha=\frac{\eps^2}{2 \sqrt{2}}$ and consider the $\alpha$-lazy versions of $P$ and $\bar{P}$:
    $$P'= \alpha I + \left(1 - \alpha \right)P,$$ $$\bar{P}' =\alpha I+\left(1 - \alpha \right)\bar{P},$$ respectively. Then $\Distance(P',\bar{P}')\geq\frac{\eps}{2}$.
    \item\label{112B} It holds
$$\Distance(P^*,\bar{P}^*)=\Distance(P,\bar{P}).$$

\item\label{112C} Assume that $P,\bar{P}$ are reversible and that \begin{equation*} \left|\left|\frac{\pi}{\bar{\pi}}-\mathbf{1}\right|\right|_{\infty}<\eps.\end{equation*} Let $\alpha\in[0,1]$ and denote $P(\alpha)=\alpha P+(1-\alpha)\bar{P}$. Then
\begin{equation*}
    \begin{split}
    \Distance(P,P(\alpha)) &\geq1-\sqrt{\alpha}-2\sqrt{\frac{1-\alpha}{1-\eps}}\\&+2\sqrt{\frac{1-\alpha}{1-\eps}}\Distance(P,\bar{P}).
    \end{split}
\end{equation*}

\item\label{112D} Let $k \in \N$. Then \begin{equation*}1-(1-\Distance(P,\bar{P}))^{k}\geq\Distance(P^{k},\bar{P}^{k}).\end{equation*}

\item \label{112D_}
$$\Distance(P^k,\bar{P}^k)\underset{k\to\infty}{\longrightarrow}d_\textnormal{Hel}^2(\pi,\bar{\pi}).$$

\item\label{112E} It holds  $$\Distance(P^\dagger,\bar{P}^\dagger)\leq 2\Distance(P,\bar{P}).$$

\item\label{112F} There exist irreducible and reversible Markov chains that are arbitrarily close under the distance such that the Hellinger distance between their stationary distributions is bounded away from $0$.
\end{enumerate}
\end{proposition}

\begin{remark}

Property (\ref{112A}) exhibits the robustness of the distance under transitions to $\alpha$-lazy versions. This allows us to handle arbitrary irreducible Markov chains, although  we invoke the concentration bounds of \citet{paulin2015concentration} which hold only for ergodic Markov chains.

Property (\ref{112C}) addresses \citet[Open Questions 2]{daskalakis2018testing} asking how the distance between two Markov chains changes when one substitutes one of them with a convex combination of both.

Property (\ref{112D}) addresses \citet[Open Questions 3]{daskalakis2018testing} asking how the distance between two Markov chains is related to the distance between the same Markov chains being observed only at intervals of size $k$. 

Property (\ref{112D_}) is related to the previous one and shows that with increasing $k$ it becomes increasingly harder to distinguish between two Markov chains that have the same stationary distribution.

Reversible Markov chains enjoy favourable properties. One way to make an irreducible Markov reversible is by moving to its multiplicative reversibilization. It is therefore natural to ask how the distance behaves under this operation. This is addressed in  (\ref{112E}). 

Finally, property (\ref{112F}) reflects the fact that if two reducible Markov chains share an identical essential communicating  class, the distance between them is $0$ (cf. \citet[Claim 1]{daskalakis2018testing}).

\end{remark}

\section{CONCLUSION}
\label{section:conclusion}

In this work we have replaced the restrictive symmetry assumption made by \citet{daskalakis2018testing} and \citet{pmlr-v99-cherapanamjeri19a} with the more natural one of reversibility and showed that it is possible to perform identity testing in this class under the distance notion between Markov chains introduced by \citet{daskalakis2018testing}, provided that the stationary distributions of the reference and of the unknown Markov chains are not too far from each other under a certain distance notion between probability distributions. In addition, we provided intuition regarding the distance notion between Markov chains by making statements on its behaviour under several natural operations. The next step in our research agenda in identity testing of Markov chains is to investigate the possibility of removing the assumption on the closeness of the stationary distributions of the reference and the unknown chain Markov chains.

\bibliography{bibliography}

\begin{thebibliography}{31}
\providecommand{\natexlab}[1]{#1}
\providecommand{\url}[1]{\texttt{#1}}
\expandafter\ifx\csname urlstyle\endcsname\relax
  \providecommand{\doi}[1]{doi: #1}\else
  \providecommand{\doi}{doi: \begingroup \urlstyle{rm}\Url}\fi

\bibitem[Aldous and Diaconis(1987)]{aldous1987strong}
D.~Aldous and P.~Diaconis.
\newblock Strong uniform times and finite random walks.
\newblock \emph{Advances in Applied Mathematics}, 8\penalty0 (1):\penalty0
  69--97, 1987.

\bibitem[Brualdi(1977)]{brualdi1977introductory}
R.~A. Brualdi.
\newblock \emph{Introductory combinatorics}.
\newblock Pearson Education India, 1977.

\bibitem[Chan et~al.(2021)Chan, Ding, and Li]{chan2021learning}
S.~Chan, Q.~Ding, and S.~Li.
\newblock Learning and testing irreducible {M}arkov chains via the $k$-cover
  time.
\newblock In \emph{Algorithmic Learning Theory}, pages 458--480. PMLR, 2021.

\bibitem[Cherapanamjeri and Bartlett(2019)]{pmlr-v99-cherapanamjeri19a}
Y.~Cherapanamjeri and P.~L. Bartlett.
\newblock Testing symmetric {M}arkov chains without hitting.
\newblock In \emph{Proceedings of the Thirty-Second Conference on Learning
  Theory}, volume~99 of \emph{Proceedings of Machine Learning Research}, pages
  758--785. PMLR, 2019.

\bibitem[Daskalakis et~al.(2017)Daskalakis, Dikkala, and
  Gravin]{Daskalakis2017TestingFO}
C.~Daskalakis, N.~Dikkala, and N.~Gravin.
\newblock Testing from one sample: Is the casino really using a riffle shuffle?
\newblock \emph{ArXiv}, abs/1704.06850, 2017.

\bibitem[Daskalakis et~al.(2018{\natexlab{a}})Daskalakis, Dikkala, and
  Gravin]{daskalakis2018testing}
C.~Daskalakis, N.~Dikkala, and N.~Gravin.
\newblock Testing symmetric {M}arkov chains from a single trajectory.
\newblock In \emph{Conference On Learning Theory}, pages 385--409. PMLR,
  2018{\natexlab{a}}.

\bibitem[Daskalakis et~al.(2018{\natexlab{b}})Daskalakis, Kamath, and
  Wright]{daskalakis2018distribution}
C.~Daskalakis, G.~Kamath, and J.~Wright.
\newblock Which distribution distances are sublinearly testable?
\newblock In \emph{Proceedings of the Twenty-Ninth Annual ACM-SIAM Symposium on
  Discrete Algorithms}, pages 2747--2764. SIAM, 2018{\natexlab{b}}.

\bibitem[Drnov{\v{s}}ek and Peperko(2006)]{drnovvsek2006inequalities}
R.~Drnov{\v{s}}ek and A.~Peperko.
\newblock Inequalities for the {H}adamard weighted geometric mean of positive
  kernel operators on {B}anach function spaces.
\newblock \emph{Positivity}, 10\penalty0 (4):\penalty0 613--626, 2006.

\bibitem[Dubhashi and Panconesi(2009)]{dubhashi2009concentration}
D.~P. Dubhashi and A.~Panconesi.
\newblock \emph{Concentration of measure for the analysis of randomized
  algorithms}.
\newblock Cambridge University Press, 2009.

\bibitem[Fill(1991)]{fill1991eigenvalue}
J.~A. Fill.
\newblock Eigenvalue bounds on convergence to stationarity for nonreversible
  {M}arkov chains, with an application to the exclusion process.
\newblock \emph{The annals of applied probability}, pages 62--87, 1991.

\bibitem[Fried and Wolfer(2021)]{fried2021alpha}
S.~Fried and G.~Wolfer.
\newblock On the $\alpha$-lazy version of {M}arkov chains in estimation and
  testing problems.
\newblock \emph{arXiv preprint arXiv:2105.09536}, 2021.

\bibitem[Gibbs and Su(2002)]{gibbs2002choosing}
A.~L. Gibbs and F.~E. Su.
\newblock On choosing and bounding probability metrics.
\newblock \emph{International statistical review}, 70\penalty0 (3):\penalty0
  419--435, 2002.

\bibitem[Gohberg and Kre{\u\i}n(1978)]{gohberg1978introduction}
I.~Gohberg and M.~G. Kre{\u\i}n.
\newblock \emph{Introduction to the theory of linear nonselfadjoint operators},
  volume~18.
\newblock American Mathematical Soc., 1978.

\bibitem[Helmberg(2008)]{helmberg2008introduction}
G.~Helmberg.
\newblock \emph{Introduction to spectral theory in {H}ilbert space}.
\newblock Courier Dover Publications, 2008.

\bibitem[Horn and Johnson(2012)]{horn2012matrix}
R.~A. Horn and C.~R. Johnson.
\newblock \emph{Matrix analysis}.
\newblock Cambridge university press, 2012.

\bibitem[Kazakos(1978)]{kazakos1978bhattacharyya}
D.~Kazakos.
\newblock The {B}hattacharyya distance and detection between {M}arkov chains.
\newblock \emph{IEEE Transactions on Information Theory}, 24\penalty0
  (6):\penalty0 747--754, 1978.

\bibitem[Kemeny et~al.(2012)Kemeny, Snell, and Knapp]{kemeny2012denumerable}
J.~G. Kemeny, J.~L. Snell, and A.~W. Knapp.
\newblock \emph{Denumerable {M}arkov chains: with a chapter of {M}arkov random
  fields by David Griffeath}, volume~40.
\newblock Springer Science \& Business Media, 2012.

\bibitem[Levin and Peres(2017)]{levin2017markov}
D.~A. Levin and Y.~Peres.
\newblock \emph{{M}arkov chains and mixing times}, volume 107.
\newblock American Mathematical Soc., 2017.

\bibitem[L\'evy(1951)]{levy1951}
P.~L\'evy.
\newblock Syst\`emes {M}arkoviens et stationnaires. cas d\'enombrable.
\newblock \emph{Annales scientifiques de l'\'Ecole Normale Sup\'erieure}, 3e
  s{\'e}rie, 68:\penalty0 327--381, 1951.

\bibitem[L\'evy(1952)]{levy1952}
P.~L\'evy.
\newblock Compl\'ement \`a l'\'etude des processus de {M}arkoff.
\newblock \emph{Annales scientifiques de l'\'Ecole Normale Sup\'erieure}, 3e
  s{\'e}rie, 69:\penalty0 203--212, 1952.

\bibitem[L\'evy(1958)]{levy1958}
P.~L\'evy.
\newblock Processus {M}arkoviens et stationnaires. cas d\'enombrable.
\newblock \emph{Annales de l'institut Henri Poincar\'e}, 16\penalty0
  (1):\penalty0 7--25, 1958.

\bibitem[Linial et~al.(1995)Linial, London, and Rabinovich]{linial1995geometry}
N.~Linial, E.~London, and Y.~Rabinovich.
\newblock The geometry of graphs and some of its algorithmic applications.
\newblock \emph{Combinatorica}, 15\penalty0 (2):\penalty0 215--245, 1995.

\bibitem[Massart(2007)]{massart1896concentration}
P.~Massart.
\newblock Concentration inequalities and model selection. {L}ectures from the
  33rd summer school on probability theory held in {S}aint-{F}lour, {J}uly
  6--23, 2003.
\newblock \emph{Lecture Notes in Mathematics}, 1896, 2007.

\bibitem[Matsui et~al.(2022)Matsui, Azzaoui, and Murakami]{matsui2022analysis}
T.~Matsui, N.~Azzaoui, and D.~Murakami.
\newblock Analysis of {C}{O}{V}{I}{D}-19 evolution based on testing closeness
  of sequential data.
\newblock \emph{Japanese Journal of Statistics and Data Science}, pages 1--18,
  2022.

\bibitem[Olver et~al.(2006)Olver, Shakiban, and Shakiban]{olver2006applied}
P.~J. Olver, C.~Shakiban, and C.~Shakiban.
\newblock \emph{Applied linear algebra}, volume~1.
\newblock Springer, 2006.

\bibitem[Paulin(2015)]{paulin2015concentration}
D.~Paulin.
\newblock Concentration inequalities for {M}arkov chains by {M}arton couplings
  and spectral methods.
\newblock \emph{Electron. J. Probab.}, 20:\penalty0 32 pp., 2015.
\newblock \doi{10.1214/EJP.v20-4039}.

\bibitem[Sinclair and Jerrum(1989)]{sinclair1989approximate}
A.~Sinclair and M.~Jerrum.
\newblock Approximate counting, uniform generation and rapidly mixing {M}arkov
  chains.
\newblock \emph{Information and Computation}, 82\penalty0 (1):\penalty0
  93--133, 1989.

\bibitem[Waggoner(2015)]{waggoner2015lp}
B.~Waggoner.
\newblock Lp testing and learning of discrete distributions.
\newblock In \emph{Proceedings of the 2015 Conference on Innovations in
  Theoretical Computer Science}, pages 347--356, 2015.

\bibitem[Wolfer and Kontorovich(2020)]{pmlr-v108-wolfer20a}
G.~Wolfer and A.~Kontorovich.
\newblock Minimax testing of identity to a reference ergodic {M}arkov chain.
\newblock In \emph{Proceedings of the Twenty Third International Conference on
  Artificial Intelligence and Statistics}, volume 108 of \emph{Proceedings of
  Machine Learning Research}, pages 191--201. PMLR, 26--28 Aug 2020.

\bibitem[Zhan(2013)]{zhan2013matrix}
X.~Zhan.
\newblock \emph{Matrix theory}, volume 147.
\newblock American Mathematical Soc., 2013.

\bibitem[Zhao and Liu(1996)]{zhao1996censored}
Y.~Q. Zhao and D.~Liu.
\newblock The censored {M}arkov chain and the best augmentation.
\newblock \emph{Journal of Applied Probability}, pages 623--629, 1996.

\end{thebibliography}
\bibliographystyle{plainnat}

\appendix
\section*{Appendix}
\section{PROOF OF LEMMA~\ref{lem; 225}} \label{app; 1}

Suppose $P$ is reversible, i.e. $P$ satisfies the detailed balance equation (e.g. \citet[(1.30)]{levin2017markov}: $$Q(i,j)=Q(j,i),\;\;\forall i,j\in [d].$$ Assume without loss that $S=[k]$ for some $k\in[d]$ and let $i,j\in[k]$. We have
\begin{align}
\pi_S(i)P_{\textnormal{cen}}(S)(i,j)=&\frac{\pi(i)}{\pi(S)}\left(P_S+\sum_{t=1}^{\infty}P_{S,[d]\setminus S}P_{[d]\setminus S]}^tP_{[d]\setminus S,S}\right)(i,j)\nonumber\\=&\frac{1}{\pi(S)}\left(\pi(i)P_{S}(i,j)+\pi(i)\left(\sum_{t=1}^{\infty}P_{S,[d]\setminus S}P_{[d]\setminus S}^tP_{[d]\setminus S,S}\right)(i,j)\right)\nonumber\\=&\frac{1}{\pi(S)}\left(\pi(j)P_{S}(j,i)+\sum_{t=1}^{\infty}\left(\pi(i)P_{S,[d]\setminus S}P_{[d]\setminus S}^tP_{[d]\setminus S,S}\right)(i,j)\right).\label{eq; 54}
\end{align}
Now, let $t\in\N$. We have
\begin{align}
\left(\pi(i)P_{S, [d]\setminus S}P_{[d]\setminus S}^tP_{[d]\setminus S,S}\right)(i,j)=&\sum_{l=k+1}^{d}\sum_{m=k+1}^{d}\pi(i)P(i,l)P^t(l,m)P(m,j)\nonumber\\=&\pi(j)\sum_{l=k+1}^{d}\sum_{m=k+1}^{d}P(l,i)P^t(m,l)P(j,m)\nonumber\\=&\left(\pi(j)P_{S, [d]\setminus S}P_{[d]\setminus S}^tP_{[d]\setminus S,S}\right)(j,i)\nonumber.
\end{align}
Thus,
\begin{align} 
(\ref{eq; 54}) =  & \frac{1}{\pi(S)}\left(\pi(j)P_{S}(j,i)+\sum_{t=1}^{\infty}\left(\pi(j)P_{S,[d]\setminus S}P_{[d]\setminus S}^tP_{[d]\setminus S,S}\right)(j,i)\right)\nonumber\\=& \pi_S(j)P_{\textnormal{cen}}(S)(j,i)\nonumber. 
\end{align}
\qed

\section{PROOF OF LEMMA~\ref{lem; 537}}
\label{section:proof-min-max}
It is easy to see that $\sqrt{P \circ \bar{P}}$ is self-adjoint in $L_2\left(\sqrt{\pi \circ \bar{\pi}}\right)$. Let $\mathbf{1}_S\in\R^d$ be given by $$\mathbf{1}_S(i)=\begin{cases}1&i\in S \\0 & \textnormal{otherwise}, \end{cases} \;\forall i\in[d].$$ Let $\bar{Q}=\diag(\bar{\pi})\bar{P}$ be the edge measure and let $R=\diag(\bar{\pi})P$. Then 
\begin{equation*}
    \begin{split}
        \rho\left(\sqrt{P\circ \bar{P}}\right) &= \max_{u\neq 0}\frac{\langle \sqrt{P \circ \bar{P}}u,u\rangle_{\sqrt{\pi\circ \bar{\pi}}}}{||u||_{2,\sqrt{\pi \circ \bar{\pi}}}^2} \\ &\stackrel{(i)}{\geq} \frac{\langle \sqrt{P \circ \bar{P}}\mathbf{1}_S,\mathbf{1}_S\rangle_{\sqrt{\pi\circ \bar{\pi}}}}{||\mathbf{1}_S||_{2,\sqrt{\pi \circ \bar{\pi}}}^2}\\&= \frac{\sum_{i,j\in S}\sqrt{R(i,j)}\sqrt{\bar{Q}(i,j)}\sqrt{\frac{\pi(i)}{\bar{\pi}(i)}}}{\sum_{i\in S}\sqrt{\pi(i)}\sqrt{\bar{\pi}(i)}}
        \\&=
\frac{\sum_{i,j\in S}\sqrt{\frac{R(i,j)}{\bar{\pi}(S)}} \sqrt{\frac{\bar{Q}(i,j)}{\bar{\pi}(S)}}\sqrt{\frac{\pi(i)}{\bar{\pi}(i)}}}{\sum_{i\in S}\sqrt{\frac{\pi(i)}{\pi(S)}}\sqrt{\frac{\bar{\pi}(i)}{\bar{\pi}(S)}}\sqrt{\frac{\pi(S)}{\bar{\pi}(S)}}} \\
\\&\stackrel{(ii)}{\geq}
 \sum_{i,j\in S}\sqrt{\frac{R(i,j)}{\bar{\pi}(S)}} \sqrt{\frac{\bar{Q}(i,j)}{\bar{\pi}(S)}}\sqrt{\frac{\pi(i)}{\bar{\pi}(i)}} \sqrt{\frac{\bar{\pi}(S)}{\pi(S)}}\\
 &\stackrel{(iii)}{\geq}
 \sum_{i,j\in S}\sqrt{\frac{R(i,j)}{\bar{\pi}(S)}} \sqrt{\frac{\bar{Q}(i,j)}{\bar{\pi}(S)}}\underbrace{\sqrt{1 - \frac{\eps}{2}}\sqrt{1 + \frac{\eps}{2}}}_{\geq 1 - \frac{\eps}{2}},
    \end{split}
\end{equation*}
where $(i)$ is due to the Courant–Fischer principle (e.g. \citet[p. 219]{helmberg2008introduction}), $(ii)$ stems from the AM-GM inequality as follows: $$\sum_{i\in S}\sqrt{\frac{\pi(i)}{\pi(S)}}\sqrt{\frac{\bar{\pi}(i)}{\bar{\pi}(S)}}\leq \frac{1}{2}\sum_{i\in S}\left(\frac{\pi(i)}{\pi(S)}+\frac{\bar{\pi}(i)}{\bar{\pi}(S)}\right)= 1$$
and $(iii)$ follows by definition of $\pi(S)$ and the assumption that $\left|\left|\frac{\pi}{\bar{\pi}}-\mathbf{1}\right|\right|_{\infty}\leq\frac{\eps}{2}$:
$$\pi(S) = \sum_{i \in S} \pi(i) \leq \left(1 + \frac{\eps}{2}\right) \sum_{i \in S} \bar{\pi}(i) = \left(1 + \frac{\eps}{2}\right) \bar{\pi}(S).$$
By assumption, $\textnormal{Distance}(P, \bar{P}) \geq \eps$. Thus,
\begin{equation*}
    \begin{split}
        \eps &\leq 1 - \left(1- \frac{\eps}{2}\right) \sum_{i,j\in S}\sqrt{\frac{R(i,j)}{\bar{\pi}(S)}} \sqrt{\frac{\bar{Q}(i,j)}{\bar{\pi}(S)}}.
    \end{split}
\end{equation*}
It follows
\begin{equation}
\label{eq; 34}
    \frac{\eps}{2} \leq \left(1 - \frac{\eps}{2}\right)\left(1-  \sum_{i,j\in S}\sqrt{\frac{R(i,j)}{\bar{\pi}(S)}} \sqrt{\frac{\bar{Q}(i,j)}{\bar{\pi}(S)}}\right)\leq 1-  \sum_{i,j\in S}\sqrt{\frac{R(i,j)}{\bar{\pi}(S)}} \sqrt{\frac{\bar{Q}(i,j)}{\bar{\pi}(S)}}.
\end{equation}
We distinguish between two cases: First assume that \begin{equation}\label{eq; 77}\sum_{i,j\in S} \textnormal{Distribution}(S, P,\bar{\pi})(i,j)\geq 1-\frac{5\eps}{16}.\end{equation} In this case we have
\begin{align}
d^2_\textnormal{Hel}(\textnormal{Distribution}(S, P,\bar{\pi}), \;& \textnormal{Distribution}(S, \bar{P}, \bar{\pi}))\nonumber\\ \geq&\frac{1}{2}\sum_{i,j\in S} \left(\sqrt{\textnormal{Distribution}(S, P,\bar{\pi})(i,j)}-\sqrt{ \textnormal{Distribution}(S, \bar{P}, \bar{\pi})(i,j)}\right)^2\nonumber\\=& \frac{1}{2}\sum_{i,j\in S} \left(\sqrt{\frac{R(i,j)}{\bar{\pi}(S)}}- \sqrt{\frac{\bar{Q}(i,j)}{\bar{\pi}(S)}}\right)^2\nonumber\\=&\frac{1}{2}\sum_{i,j\in S}\left(\frac{R(i,j)}{\bar{\pi}(S)}+ \frac{\bar{Q}(i,j)}{\bar{\pi}(S)}-2\sqrt{\frac{R(i,j)}{\bar{\pi}(S)}} \sqrt{\frac{\bar{Q}(i,j)}{\bar{\pi}(S)}}\right)\nonumber\\\geq& 1-\frac{3\eps}{16}-\sum_{i,j\in S}\sqrt{\frac{R(i,j)}{\bar{\pi}(S)}}\sqrt{\frac{\bar{Q}(i,j)}{\bar{\pi}(S)}}\nonumber\\\geq&\eps-\frac{3\eps}{16}=\frac{13\eps}{16}\nonumber\geq\frac{\eps^2}{128}
\end{align} where in the second inequality we used (\ref{eq; 77}), (\ref{eq; 34}) and that, by assumption,  $$\sum_{i,j\in S} \textnormal{Distribution}(S, \bar{P},\bar{\pi})(i,j)\geq 1-\frac{\eps}{16}.$$

Consider now the case $$\sum_{i,j\in S} \Distribution(S, P, \bar{\pi})(i,j)\leq  1-\frac{5\eps}{16}.$$ In this case we have \begin{align}
d_{\textnormal{TV}}(\textnormal{Distribution}(S, P, \bar{\pi}),\textnormal{Distribution}(S, \bar{P},\bar{\pi}))\geq&\frac{1}{2} \sum_{i,j\in S}\left(\frac{\bar{Q}(i,j)}{\bar{\pi}(S)}- \frac{R(i,j)}{\bar{\pi}(S)}\right)\nonumber\\\geq &\frac{1}{2}\left(1-\frac{\eps}{16}-1+\frac{5\eps}{16}\right)=\frac{\eps}{16}.\nonumber
\end{align} By (\ref{eq; 46}), $$d^2_\textnormal{Hel}(\textnormal{Distribution}(S, P, \bar{\pi}),\textnormal{Distribution}(S, \bar{P},\bar{\pi}))\geq\frac{\eps^2}{128}.$$
\qed

\section{PROOF OF LEMMA~\ref{lem; 32}}
\label{app; 2}


Denote $m=|T|$ and assume without loss that $T=[m]$.
There exists a non negative left Perron vector $u\in\R^m$ corresponding to $\lambda$ which we may assume to be not descending, i.e., $$u(1)\leq u(2)\leq\cdots\leq u(m).$$ By abuse of notation we denote by $u$ the vector in $\R^d$ obtained from $u$ by extending it with $d-m$ zeros. Now, define $\hat{u}\in\R^d$ by 
$\hat{u}(i)= \frac{u(i)}{\pi(i)}, \forall i\in[d]$. It holds 

\begin{equation}\label{eq; 1}
\langle u(I-P),\hat{u}\rangle=(1-\lambda)\langle u,\hat{u}\rangle.
\end{equation}
The right-hand side of (\ref{eq; 1}) equals $$(1-\lambda)\sum_{i\in[d]}\pi(i)\hat{u}(i)^2$$ while the left-hand side is bounded below by $$\sum_{1\leq i<j\leq d}Q(i,j)\left(\hat{u}(i)-\hat{u}(j)\right)^2.$$ Thus, $$1-\lambda\geq\frac{\sum_{1\leq i<j\leq d}Q(i,j)\left(\hat{u}(i)-\hat{u}(j)\right)^2}{\sum_{i\in[d]}\pi(i)\hat{u}(i)^2}.$$
Now, $$\sum_{1\leq i<j\leq d}Q(i,j)\left(\hat{u}(i)+\hat{u}(j)\right)^2\leq2\sum_{i\in[d]}\pi(i)\hat{u}(i)^2.$$ It follows that $$1-\lambda\geq\frac{1}{2}\left(\frac{\sum_{1\leq i<j\leq d}Q(i,j)\left(\hat{u}(i)^2-\hat{u}(j)^2\right)}{\sum_{i\in[d]}\pi(i)\hat{u}(i)^2}\right)^2.$$ 
Now, for $1\leq k\leq d-1$ let $S_k = [k]$. Then 
\begin{align}
\sum_{1\leq i<j\leq d}Q(i,j)\left(\hat{u}(i)^2-\hat{u}(j)^2\right)=&\sum_{k=1}^m\left(\hat{u}(k)^2-\hat{u}(k+1)^2\right)\sum_{i\in S_k, j\in\bar{S_k}}Q(i,j)\nonumber\\ \geq &\alpha\sum_{k=1}^m\left(\hat{u}(k)^2-\hat{u}(k+1)^2\right)\sum_{l=1}^k\pi(l)\nonumber \\ = & \alpha\sum_{l=1}^m\pi(l)\sum_{k=l}^m\left(\hat{u}(k)^2-\hat{u}(k+1)^2\right)\nonumber\\ = & \alpha\sum_{l=1}^d\pi(l)\hat{u}(l)^2.\nonumber
\end{align}
\qed

\section{PROOF OF LEMMA~\ref{lem; 890}}
\label{section:proof-lemma-890}
Let $\pi_T$ be the vector obtained from $\pi$ by keeping only the entries at indices belonging to $T$ and let $\lambda$ be the largest eigenvalue of $P_T$. By Lemma \ref{lem; 32}, $\lambda\leq 1-\frac{\alpha^2}{2}$. Let $k=\frac{8\log\frac{1}{(\pi_T)_\star}}{\alpha^2}$ and let $j\in[k]$. Define $$Y_j = 1\{\exists i\in [(j-1)k+1,jk]\;|\; X_i\notin T\}.$$ We distinguish between two cases:
\begin{enumerate}
    \item $X_{(j-1)k+1}\notin T$: In this case, $\PP(Y_j=1\;|\;X_1,\ldots,X_{(j-1)k+1})=1$.
\item $X_{(j-1)k+1}=i\in T$: In this case, with $D=\diag(\pi_T)$ where  \begin{align}
\PP(Y_j=0\;|\;X_1,\ldots,X_{(j-1)k+1})=&e_i^TP_T^{k-1}\mathbf{1}\nonumber \\ = &e_i^TD^{-1}DP_T^{k-1}\mathbf{1}\nonumber\\=&\left\langle P_T^{k-1}\mathbf{1},D^{-1}e_i\right\rangle_{\pi_T}\nonumber \\ = &\sqrt{\frac{\pi(T)}{\pi(i)}}\left\langle P_T^{k-1}\frac{1}{\sqrt{\pi(T)}}\mathbf{1},\sqrt{\pi(i)}D^{-1}e_i\right\rangle_{\pi_T}.\label{num; 1}
\end{align}
Notice that $\left|\left|\frac{1}{\sqrt{\pi(T)}}\mathbf{1}\right|\right|_{\pi_T}=\left|\left|\sqrt{\pi(i)}D^{-1}e_i\right|\right|_{\pi_T}=1$.  
By the Courant–Fischer principle (e.g. \citet[p. 219]{helmberg2008introduction}),
$$(\ref{num; 1}) \leq \frac{1}{\sqrt{(\pi_T)_\star}}\left(1-\frac{\alpha^2}{2}\right)^{k-1}\leq\frac{1}{2}$$ where in the first inequality we also used that $\sqrt{\frac{\pi(T)}{\pi(i)}}\leq \frac{1}{\sqrt{(\pi_T)_\star}}$ and in the second the definition of $k$.
\end{enumerate}
Combining the two cases we obtain, $\PP(Y_j=1\;|\;X_1,\ldots,X_{(j-1)k})\geq\frac{1}{2}.$ Thus, 
$$\PP\left(\sum_{i=1}^{m}1\{X_i\notin T\}\geq\frac{m}{4k}\right) \geq\PP\left(\sum_{i=1}^{m/k}Y_i\geq\frac{m}{4k}\right)\geq 1-\delta$$ where in the second inequality we used an adaptation of Hoeffding's inequality stated and proved in the following lemma.
\qed

\begin{lemma}
Let $\alpha\geq 0$ and let $(B_t)_{t\in\N}$ be Bernoulli random variables, not necessarily independent, such that for every $t\in\N$ it holds
$$\min_{b_1, \ldots, b_{t-1}\in\{0,1\}} \E{B_t \;|\; B_1 = b_1, \dots, B_{t-1} = b_{t-1}} \geq \alpha.$$ Then $$\PR{\sum_{t = 1}^{n} B_t \leq \frac{\alpha n}{2}} \leq e^{ -\frac{\alpha^2n}{2}}.$$
\end{lemma}

\begin{proof}
Let $n\in\N$ and $\lambda>0$. By Markov's inequality, 
$$\PR{\sum_{t = 1}^{n} B_t \leq \frac{\alpha n}{2}} \leq e^{ \frac{\lambda \alpha n}{2}} \E{\prod_{t = 1}^{n} e^{-\lambda B_t}}.$$
Now, 
\begin{equation*}
\begin{split}
    \E{\prod_{t = 1}^{n} e^{-\lambda B_t}} =& \sum_{b_1, \ldots, b_n\in\{0,1\}} \prod_{t = 1}^{n} e^{-\lambda  b_t} \PR{B_1 = b_1, \dots, B_n = b_n} \\
    =& \sum_{b_1\in\{0,1\}} e^{- \lambda b_1} \PR{B_1 = b_1} \sum_{b_2\in\{0,1\}} e^{-\lambda  b_2} \PR{B_2 = b_2 \;|\; B_1 = b_1} \cdots \\
    & \hspace{130pt} \sum_{b_n\in\{0,1\}} e^{-\lambda b_n} \PR{B_n = b_n\; |\;B_1 = b_1, \ldots,  B_{n-1} = b_{n-1}}.
\end{split}
\end{equation*}
For the last term in the above equality it holds $$
\sum_{b_n\in\{0,1\}} e^{-\lambda b_n} \PR{B_n = b_n\; |\;B_1 = b_1, \ldots,  B_{n-1} = b_{n-1}} = \E[B_n\; |\;B_{1} = b_{1}, \ldots, B_{n-1} = b_{n-1}]{e^{-\lambda B_n}} \leq  e^{-\lambda\alpha + \frac{\lambda^2}{8}}$$ where the inequality is due to Hoeffding's lemma (e.g. \citet[Lemma 2.6]{massart1896concentration}). Proceeding inductively, we obtain \begin{equation*}
    \E{\prod_{t = 1}^{n} e^{-\lambda  B_t}} \leq  \left(e^{ -\lambda \alpha +  \frac{\lambda^2}{8} }\right)^n.
\end{equation*}
Thus, $$\PR{\sum_{t = 1}^{n} B_t \leq \frac{\alpha n}{2}} \leq\left(e^{-\lambda\alpha+\frac{\lambda^{2}}{4}}\right)^{n/2}\leq e^{-\frac{\alpha^{2}n}{2}}.$$ where in the second inequality we used that $-\lambda\alpha +  \frac{\lambda^2}{4}$ attains a minimum of $-\alpha^{2}$ at $\lambda=2\alpha$. 

\end{proof}

\section{PROOF OF LEMMA~\ref{lem; 1}}
\label{section:proof-paulin}

Let $m\in\N$. Consider first the stationary case, i.e., $\mu=\pi$.  By \citet[Theorem 3.4]{paulin2015concentration}, $$\PP\left(|N_m(i)-m\pi(i)|\geq \frac{1}{2}\pi(i)\right)\leq 2\exp\left(-\frac{m\gamma\pi(i)}{36}\right).$$ 
Thus, there exists a universal constant $c'$ such that if $m>\frac{c'\log \frac{d}{\delta}}{\pi(i)\gamma}$ then $$\PP\left(N_m(i)\notin\left(\frac{1}{2}m\pi(i),\frac{3}{2}m\pi(i)\right)\right)<\frac{\delta}{d}.$$ Now, in order to accommodate for non-stationary chains,
by \citet[Proposition 3.10]{paulin2015concentration}, we need to replace $\log \frac{d}{\delta}$ with $$\log \frac{d||\mu/\pi||_{2,\pi}}{\delta}\leq \log\frac{d}{\delta\pistar}\leq 2\log\frac{1}{\delta\pistar}.$$ 
Replacing $\pi(i)$ with $\pistar$ and using the union bound proves the assertion.
\qed

\section{PROOF OF LEMMA~\ref{lem; 522}}
\label{522}

It holds 
\begin{align}
\PP\left(\exists i\in[d] \textnormal{ s.t } v(i)>2mp(i)\right)\leq& d\max_{i\in[d]}\PP(v(i)>2mp(i)))\nonumber\\\leq& d\exp\left(-\frac{m^{2}p_\star^{2}}{2\left(mp_\star(1-p_\star)+\frac{mp_\star}{3}\right)}\right)\nonumber\\ \leq& d\exp\left(-\frac{mp_\star}{4}\right)<\delta\nonumber
\end{align}
where the second inequality is due to Bernstein's inequality (e.g. \citet[Theorem 1.2]{dubhashi2009concentration}).
\qed

\section{PROOF OF PROPOSITION~\ref{prop, 112} }\label{sss}

\begin{enumerate} 
\item [(\ref{112A})]
Let $i,j\in[d]$. It holds 
\begin{align}
\left(\sqrt{P'\circ\bar{P}'}\right)(i,j)=&\sqrt{ \left( \left(1 - \alpha \right)P(i,j) + \alpha 1\{i = j\} \right) \left( \left(1 - \alpha \right)\bar{P}(i,j) + \alpha 1\{i = j\} \right) }\nonumber\\=&
 \sqrt{ \left(1 - \alpha \right)^2P(i,j) \bar{P}(i,j) + \alpha 1\{i = j\} \left( (1 - \alpha)P(i,j) + (1 - \alpha)\bar{P}(i,j) + \alpha \right)}\nonumber\\\leq & \sqrt{P(i,j) \bar{P}(i,j)} + \sqrt{2\alpha} 1\{i = j\}\nonumber\\=&\left(\sqrt{P\circ\bar{P}} + \sqrt{2\alpha}I\right)(i,j).\nonumber
\end{align}
Now, $$\rho\left(\sqrt{P'\circ\bar{P}'}\right)\leq\rho\left(\sqrt{P\circ\bar{P}} + \sqrt{2\alpha}I\right)=\rho\left(\sqrt{P^\circ\bar{P}}\right)+\sqrt{2\alpha}$$ where the inequality is due to the monotonicity of the spectral radius (e.g. \citet[Theorem 8.1.18]{horn2012matrix}). Thus, $$\Distance(P',\bar{P}')\geq \Distance(P,\bar{P})-\sqrt{2\alpha}=\frac{\eps}{2}.$$

\item [(\ref{112B})]
Using that the spectral radius is invariant under transposition and matrix similarity, 
\begin{align}
  \rho\left(\sqrt{P^{*}\circ\bar{P}^{*}}\right)=&\rho\left(\sqrt{P^{*}\circ\bar{P}^{*}}^T\right)\nonumber\\=&\rho\left(\sqrt{D_{\pi}PD_{\pi^{-1}}\circ D_{\bar{\pi}}\bar{P}D_{\bar{\pi}^{-1}}}\right)\nonumber\\=& \rho\left(D_{\sqrt{\pi\circ\bar{\pi}}}\sqrt{P\circ\bar{P}}D_{\sqrt{\pi\circ\bar{\pi}}}^{-1}\right)\nonumber\\=&\rho\left(\sqrt{P\circ\bar{P}}\right)\nonumber.
\end{align}

\item [(\ref{112C})]

Let $i,j\in[d]$. Then 
$$
\sqrt{P\circ P(\alpha)}(i,j)=\sqrt{P\circ\left(\alpha P+(1-\alpha)\bar{P}\right)}(i,j) \leq \sqrt{\alpha}P(i,j)+\sqrt{1-\alpha}\sqrt{P\circ \bar{P}}(i,j)$$ where the inequality is due to the subadditivity of the function $x\mapsto\sqrt{x}$.
By \citet[Theorem 8.1.18]{horn2012matrix}, $$\rho\left(\sqrt{P\circ P(\alpha)}\right)\leq\rho\left(\sqrt{\alpha}P+\sqrt{1-\alpha}\sqrt{P\circ \bar{P}}\right).$$
Now, with $D=\diag(\pi)^{\frac{1}{2}}$ and $\bar{D}=\diag(\bar{\pi})^{\frac{1}{2}}$, both $D^{-1}PD$ and $\sqrt{(D\bar{D})^{-1}P\circ \bar{P}(D\bar{D})}$ are symmetric. Furthermore, for every $i,j\in[d]$ it holds \begin{align}
    \left(D^{-1}\sqrt{P\circ\bar{P}}D\right)(i,j)=&\pi(j)^{-\frac{1}{2}}\sqrt{P(i,j)\bar{P}(i,j)}\pi(i)^{\frac{1}{2}}\nonumber\\=&\sqrt{\pi(j)^{-\frac{1}{2}}P(i,j)\pi(i)^{\frac{1}{2}}\bar{\pi}(j)^{-\frac{1}{2}}\bar{P}(i,j)\bar{\pi}(i)^{\frac{1}{2}}\sqrt{\frac{\pi(i)\bar{\pi}(j)}{\bar{\pi}(i)\pi(j)}}}\nonumber\\\leq&\sqrt[^4]{\frac{1+\eps}{1-\eps}}\sqrt{(D\bar{D})^{-1}P\circ \bar{P}(D\bar{D})}(i,j).\nonumber
\end{align}Thus, \begin{align}\rho\left(\sqrt{\alpha}P+\sqrt{1-\alpha}\sqrt{P\circ \bar{P}}\right)=&\rho\left(\sqrt{\alpha}D^{-1}PD +\sqrt{1-\alpha}D^{-1}\sqrt{P\circ \bar{P}}D\right)\nonumber\\\leq&\rho\left(\sqrt{\alpha}D^{-1}PD +\sqrt{1-\alpha}\sqrt[^4]{\frac{1+\eps}{1-\eps}}\sqrt{(D\bar{D})^{-1}P\circ \bar{P}(D\bar{D})}\right)\nonumber\\\leq&\sqrt{\alpha}\rho\left(D^{-1}PD\right) +\sqrt{1-\alpha}\sqrt[^4]{\frac{1+\eps}{1-\eps}}\rho\left(D^{-1}\sqrt{P\circ \bar{P}}D\right)\nonumber\\=&\sqrt{\alpha} +\sqrt{1-\alpha}\sqrt[^4]{\frac{1+\eps}{1-\eps}}\rho\left(\sqrt{P\circ \bar{P}}\right)\nonumber\end{align} where the first inequality is due to the monotonicity of the spectral radius (e.g. \citet[Theorem 8.1.18]{horn2012matrix}) and the second inequality is due to the fact that the spectral radius is subadditive for symmetric matrices (e.g. \citet[Theorem 9.21 together with Exercise 9.2.37]{olver2006applied}). 

\item [(\ref{112D})]

By monotonicity of the spectral radius (e.g. \citet[Corollary 6.15]{zhan2013matrix}), it suffices to show that for each $i,j\in[d]$ it holds $$\sqrt{P\circ \bar{P}}^{k}(i,j)\leq\sqrt{P^{k}\circ \bar{P}^{k}}(i,j).$$ 
We proceed by induction. The claim is obviously true for $k=1$. Suppose it holds for $k\in\N$. We have \begin{align}
\sqrt{P\circ \bar{P}}^{k+1}(i,j)=&\sum_{l=1}^{d}\sqrt{P\circ \bar{P}}(i,l)\sqrt{P\circ \bar{P}}^{k}(l,j)    \nonumber\\ \leq &\sum_{l=1}^{d}\sqrt{P(i,l)P^{k}(l,j)}\sqrt{\bar{P}(i,l)\bar{P}^{k}(l,j)}\nonumber\\\leq & \sqrt{\sum_{l=1}^{d}P(i,l)P^{k}(l,j)}\sqrt{\sum_{l=1}^{d}\bar{P}(i,l)\bar{P}^{k}(l,j)}\nonumber\\=& \sqrt{P_{1}^{k+1}\circ P_{2}^{k+1}}(i,j)\nonumber
\end{align}
where we used the induction hypothesis in the first inequality and Cauchy-Schwarz in the second.

\item [(\ref{112D_})] It holds $$P^k\underset{k\to\infty}{\longrightarrow} \begin{pmatrix}
- & \pi & -\\
- &\vdots & - \\
- & \pi & -
\end{pmatrix}\text{ and } \bar{P}^k\underset{k\to\infty}{\longrightarrow} \begin{pmatrix}
- & \bar{\pi} & -\\
- &\vdots & - \\
- & \bar{\pi} & -
\end{pmatrix}$$ where convergence is under the infinity norm over matrices. Thus, $$\text{Distance}(P^k,\bar{P}^k)\underset{k\to\infty}{\longrightarrow}d_\text{Hel}^2(\pi,\bar{\pi}).$$

\item [(\ref{112E})]

It holds
 \begin{align}
   \sqrt{P^\dagger\circ\bar{P}^\dagger}=&\sqrt{P^*P\circ\bar{P}^*\bar{P}}\nonumber\\ \geq& \sqrt{P^{*}\circ\bar{P}^{*}}\sqrt{P\circ\bar{P}}\nonumber\\=&\sqrt{D_\pi^{-1}P^TD_\pi\circ D_{\bar{\pi}}^{-1}\bar{P}^TD_{\bar{\pi}}}\sqrt{P\circ\bar{P}}\nonumber \\=&D_{\sqrt{\pi\circ\bar{\pi}}}^{-1}\sqrt{P\circ \bar{P}}^TD_{\sqrt{\pi\circ\bar{\pi}}}\sqrt{P\circ\bar{P}}\nonumber \\=&\sqrt{P\circ \bar{P}}^*\sqrt{P\circ\bar{P}}\nonumber \\=&\sqrt{P\circ\bar{P}}^\dagger \nonumber
\end{align}
where the first inequality is due to \citet[Theorem 2.2]{drnovvsek2006inequalities}. It follows that $$\rho\left(\sqrt{P^\dagger\circ\bar{P}^\dagger}\right)\geq\rho\left(\sqrt{P\circ\bar{P}}^\dagger\right)\geq\rho\left(\sqrt{P\circ\bar{P}}\right)^2$$ where the second inequality is due to Weyl (e.g. \citet[Lemma 3.3]{gohberg1978introduction}).

\item [(\ref{112F})] For $\alpha\in(0,1)$ define $$P= \begin{pmatrix}
1-\alpha & \alpha \\
\frac{1}{2} &\frac{1}{2}\\
\end{pmatrix}\text{ and } 
\bar{P}= \begin{pmatrix}
1-\alpha & \alpha \\
\alpha &1-\alpha\\
\end{pmatrix}.$$
Clearly, $P,\bar{P}$ are irreducible and reversible and one verifies easily that
\begin{enumerate}
    \item The stationary distributions of $P, \bar{P}$ are $\pi=\left(\frac{\frac{1}{2}}{\frac{1}{2}+\alpha},\frac{\alpha}{\frac{1}{2}+\alpha}\right), \bar{\pi}=\left(\frac{1}{2},\frac{1}{2}\right)$, respectively.
    \item $1-\sum_{i\in [2]} \sqrt{\pi(i)\bar{\pi}(i)}\underset{\alpha\to 0}{\longrightarrow}1-\frac{1}{\sqrt{2}}>\frac{1}{4}$.
    \item $\Distance(P,\bar{P})\underset{\alpha\to 0}{\longrightarrow} 0.$
\end{enumerate}

\end{enumerate}
\qed

\section{ON THE IIDGENERATOR ALGORITHM}
\label{section:q-sampler}
Let $P\in\Mirr$ with stationary distribution $\pi$ and let $\mu\in\Delta_d$.
Let $(X_t)_{t\in\N}\sim(P,\mu)$. For $i\in [d]$ we define the \emph{first hitting time for $i$} to be
$$\tau^{(i)}_1 = \inf \set{ t \geq 1 : X_t = i}$$ (cf. \citet[p. 11]{levin2017markov})
and for $1<s\in\N$ we define the \emph{$s$th hitting time for $i$} to be
$$\tau^{(i)}_s = \inf \set{ t > \tau^{(i)}_{s-1} : X_t = i}.$$
Following~\citet[pp. 12-13]{daskalakis2018testing}, for $v \in \N^d$ we define the mapping $\Psi_{v}: \{0,1,\ldots,d\}^\infty \to \prod_{i=1}^{d} [d]^{v(i)}$ by $$(X_t)_{t\in\N} \mapsto \left(\left(X_{\tau^{(1)}_{t} + 1}\right)_{t\in[v(1)]}, \left(X_{\tau^{(2)}_{t} + 1}\right)_{t\in[v(2)]},\ldots,\left(X_{\tau^{(d)}_{t} + 1}\right)_{t\in[v(d)]}\right).$$

The map outputs, when given a trajectory sampled from $P$, for each $i\in[d]$, the first state that has been visited immediately after hitting $i$, for each of the first $v(i)$ visits to $i$. It is a consequence of the Markov property that all the coordinates of $\Psi_{v}\left(\left(X\right)_{t\in\N}\right)$
are independent and that for each $i\in[d]$, the $i$th entry of the $d$-tuple consists of a sample that is iid according to the conditional distribution defined by the $i$th row of $P$. That is, for $i \in [d]$ it holds
$$\left(X_{\tau^{(i)}_t + 1}\right)_{t \in [v(i)]} \sim P(i, \cdot)^{\otimes v(i)}.$$ The mapping $\Psi_{v}$ allows us to sample from the edge measure $Q=\diag(\pi)P$ of $P$ as described in the following two-stages procedure: Let $m \in \N$ be the desired size of the sample. First sample from $m$ iid random variables $Z_1, \dots, Z_m \sim \pi^{\otimes m}$
and denote by $v$ the corresponding histogram, i.e., $v(i) = \sum_{t=1}^{m} \pred{Z_t = i}$, for every $i\in[d]$. Second, define $\Phi_{\pi} \colon [d]^\infty \to ([d] \times [d])^m $ by $$
        (X_t)_{t\in\N} \mapsto (Y_k)_{k\in[m]},$$
    
where 
$$Y_k = \left(Z_k, \left(\left(\Psi_v(X_t)_{t\in\N}\right)\left(Z_k\right)\right)\left(1 + \sum_{j = 1}^{k - 1} \pred{Z_j = Z_k}\right) \right),\;\; \forall k\in[m].$$
Then
$(Y_k)_{k\in[m]} \sim Q^{\otimes m}$.
For an infinite trajectory, the mapping $\Phi_{\pi}$ is well-defined almost surely since, by assumption, $P$ is irreducible. However, applied on a finite trajectory $(X_t)_{t=1}^n$ for some $n\in\N$, the well-definedness of $\Phi_{\pi}\left((X_t)_{t=1}^n\right)$ is a random event that strongly depends on $n$ and on the properties of the Markov chain $P$.

Notice that the procedure above may be performed with any probability distribution $\nu\in\Delta_d$ (instead of $\pi$). In this case the sampling is from $\diag(\nu)P$ (instead of $Q$). Now, recall that in Lemma \ref{lem; 356} it is assumed that one has access to $m\in\N$ iid samples from $\Distribution(S, P, \bar{\pi})$ where $\emptyset\neq S\subseteq [d]$, $\bar{\pi}$ is the stationary distribution of the reference Markov chain and $P$ is the transition matrix of the unknown Markov chain. To achieve this, we define $\nu\in\Delta_d$ as follows:
\begin{equation}\label{eq; 1233}\nu(i)=\begin{cases}\frac{\pi(i)}{\pi(S)}& i\in S\\ 0& \text{otherwise},\end{cases}\;\;\forall i\in[d]\end{equation} and apply $\Phi_\nu$ on a trajectory $(X_t)_{t\in\N}$. This gives $(Y_k)_{k\in[m]}$ in which we replace every $Y_k$ whose second coordinate does not belong to $S$ with the symbol $\infty$. The pseudocode for the procedure of sampling from $\Distribution(S, P, \bar{\pi})$ is given in \citet[Algorithm 1]{pmlr-v99-cherapanamjeri19a}. One only needs to replace the line $$``v\ \leftarrow l \text{ samples from Uniform}(T)$$ with $$``v\ \leftarrow l \text{ samples from } \nu"$$ where $\nu$ is defined in (\ref{eq; 1233}).

\section{ALGORITHM FOR MARKOV CHAIN IDENTITY TESTING}
\label{section:algorithm}

\begin{algorithm}[H]
\DontPrintSemicolon
 \SetKwData{Distribution}{Distribution}
 \SetKwFunction{Partition}{Partition}
 \SetKwFunction{iidTester}{iidTester}
 \SetKwFunction{iidGenerator}{iidGenerator}
 \SetKwFunction{StationaryDistribution}{StationaryDistribution}
 \KwIn{$d, \eps, \bar{P}, (X_t)_{t\in[m]}
   $
 }
 \KwOut{\textsc{Accept} $= 0$
or
\textsc{Reject} $= 1$
 }
 $(\mathcal{S}, T) \leftarrow \Partition([d], \Theta(\eps))$

 $\bar{\pi}  \leftarrow \StationaryDistribution(\bar{P})$

 \For{$S\in\mathcal{S}$}{
 
    $\nu \leftarrow [0]_{d}$
 
    \For{$i\in[d]$}{
      \If{$i \in S$}{
        $\nu[i] \leftarrow\bar{\pi}[i]/\sum_{j\in S} \bar{\pi}[j]$}
   }
   
   $(\bar{\pi}_S)_\star \leftarrow \max(\nu) $
   
   $l \leftarrow \tilde{\bigO} \left(1/\eps^2(\bar{\pi}_S)_\star\right)$ 
   
   $Y \leftarrow \iidGenerator((X_t)_{t\in[m]}, S, \nu,l)$ 
   
   \If{$Y \neq \textsc{False}$}{
			 \KwRet $\iidTester(Y, \Distribution(S, \bar{P}, \bar{\pi}), \Theta(\eps^2), \Theta(1/d) )$ 
		}
 }
 \KwRet \textsc{Reject}
\end{algorithm}
\begin{remark}
\begin{enumerate}
\item ``\textsc{Accept}" means that the trajectory $(X_t)_{t\in[m]}$ from the unknown Markov chain $P$ suggests that $P=\bar{P}$ and ``\textsc{Reject}" means that probably $\Distance(P,\bar{P})\geq \eps$.  
\item The $\mathrm{StationaryDistribution}$ algorithm is any algorithm that upon receiving the transition matrix of an irreducible Markov chain returns its stationary distribution.
\item The $\mathrm{Partition}$ algorithm corresponds to  \citet[Algorithms 3 and 4]{pmlr-v99-cherapanamjeri19a} and its necessary modifications are described in Section \ref{sec; alg}.
\item The assignment to $l$ follows Lemma \ref{lem; 222}.
\item The  $\mathrm{iidGenerator}$ algorithm corresponds to  \citet[Algorithm~1]{daskalakis2018distribution} and is described in full detail in Section~\ref{section:q-sampler}.
\item The $\mathrm{iidTester}$ algorithm corresponds to \citep[Algorithm 1]{daskalakis2018distribution}. See also Section \ref{sec; 788} and Lemma \ref{lem; 356}.

\end{enumerate}
\end{remark}

\end{document}